\documentclass[11pt]{article}	

\usepackage[english]{babel}								
\usepackage{amsmath,amsfonts,amsthm,amssymb,amsxtra}				
\usepackage[title]{appendix}						

\usepackage[utf8]{inputenc}

\usepackage{fancyhdr}
\numberwithin{equation}{section}					
\numberwithin{figure}{section}						
\numberwithin{table}{section}						

\providecommand{\keywords}[1]{\textbf{\textit{Keywords---}} #1}
\providecommand{\acknowledgement}[1]{\textbf{\textit{Acknowledgements---}} #1}

\newcommand{\N}{\mathbb N}  
 
\newcommand{\R}{\mathbb R}  
\newcommand{\Z}{\mathbb Z}

\newcommand{\I}{\mathbb I}
\newcommand{\J}{\mathbb J}
\newcommand{\K}{\mathbb K}

\newcommand{\T}{\mathrm{T}}
\newcommand{\V}{\mathrm{V}}

\newcommand{\1}{\mathbf{1}}
\newcommand{\0}{\mathbf{0}}
\newcommand{\bj}{\mathbf{j}}
\newcommand{\bl}{\mathbf{l}}
\newcommand{\ie}{\textsl{i.e.}~}
\newcommand{\eg}{\textsl{e.g.}~}

\newcommand{\mbf}[1]{\mathbf{#1}}
\newcommand{\rk}[1]{\mathsf{rank} \left(#1\right)}

\newcommand{\ad}{\mathrm{ad}}
\newcommand{\ds}{\displaystyle}
\newcommand{\ee}{\varepsilon}
\newcommand{\Cmin}{C\!\!-\!\min}
\newcommand{\Cmax}{C\!\!-\!\max}

\newcommand{\dd}{\mathrm{d}}

\newcommand{\ddt}{\frac{\mathrm{d}}{\mathrm{d}t}}
\newcommand{\ddu}[2]{\frac{\partial}{\partial u_{#1}^{(#2)}}}
\newcommand{\ddv}[2]{\frac{\partial}{\partial v_{#1}^{(#2)}}}
\newcommand{\ddx}{\frac{\partial}{\partial x}}
\newcommand{\ddxx}[1]{\frac{\partial}{\partial x_{#1}}}
\newcommand{\ddy}{\frac{\partial}{\partial y}}
\newcommand{\ddyy}[1]{\frac{\partial}{\partial y_{#1}}}
\newcommand{\ddth}[1]{\frac{\partial}{\partial \theta_{#1}}}

\newtheorem{thm}{Theorem}[section]  
\newtheorem{prop}{Proposition}[section]
\newtheorem{cor}{Corollary}[section]
\newtheorem{lem}{Lemma}[section]
\newtheorem{defn}{Definition}[section]
\newtheorem{rem}{Remark}[section]

\usepackage{amsbsy}

\newcommand{\ol}[1]{\overline{#1}}

\newcounter{example}[section]

\newcounter{algorithm}
\newenvironment{algo}[1][]{\refstepcounter{algorithm}\par\medskip
   \textbf{Algorithm~\thealgorithm. #1} \rmfamily}{\medskip}

\usepackage{color}

\title{\textbf{Differential Flatness by Pure Prolongation: Necessary and Sufficient Conditions
}}

\author{Jean L\'{e}vine\thanks{Centre Automatique et Systèmes (CAS), MINES Paris - PSL, 75006 Paris, France, e-mail: jean.levine@minesparis.psl.eu}}

\date{July 2023}

\begin{document}

\maketitle

 \begin{abstract}
In this article, we introduce the notion of \emph{differential flatness by pure prolongation}: loosely speaking, a system admits this property if, and only if, there exists a pure prolongation of finite order such that the prolonged system is feedback linearizable.
We obtain Lie-algebraic necessary and sufficient conditions for a general nonlinear multi-input system to satisfy this property.
These conditions are comprised of the involutivity and relative invariance of a pair of filtrations of distributions of vector fields. An algorithm computing the minimal prolongation lengths of the input channels that achieve the system linearization, yielding the associated flat outputs, is deduced. Examples that show the efficiency and computational tractability of the approach are then presented.
\end{abstract}

\keywords{differential flatness; Lie-B\"acklund isomorphism; Lie brackets; distributions of vector fields; prolongation of vector fields; dynamic feedback linearization.}

\section{Introduction}

The notions of \emph{static feedback linearization}  \cite{JR,HSM} and \emph{dynamic feedback linearization} of a nonlinear system, whose preliminary results were reported in \cite{CLMscl,CLM,Sh,ABMP-ieee} (see also \cite{Sluis-scl93,SlT-scl96,JF,BC-scl2004,FrFo-ejc2005}), were at the origin of a so far uninterrupted thread of studies in nonlinear system theory. In particular, during the last three decades, they gave birth to the concept of \emph{differential flatness}, that plays a prominent part in motion planning and trajectory tracking problems and their applications (see~\cite{pM1,FLMR-ijc,FLMR-ieee} and \cite{L_book} for a thorough presentation). 

Sufficient or necessary conditions in special cases, as well as a general differential flatness characterization, although without an upper bound on the number of iterations of the corresponding algorithm, have been obtained (see \eg \cite{L_book,L-AAECC} for a historical review). Nevertheless, the question of obtaining \emph{computationally tractable} general necessary and sufficient conditions for dynamic feedback linearization as well as for differential flatness, remains open.  In this paper, we restrict our study to the class of \emph{differentially flat systems by pure prolongation}, \ie, roughly speaking, the class of $n$-dimensional systems with $m$ inputs
$$\dot{x}=f(x,u_{1}, \ldots, u_{m})$$ 
for which there exists a multi-index $\mbf{j}\triangleq (j_{1}, \ldots, j_{m})$ such that the prolonged system
\begin{equation}\label{prolsys:eq}
\begin{aligned}
&\dot{x}=f(x,u_{1}, \ldots, u_{m}), \\
&u_{i}^{(j_{i}+1)}=v_{i},  \quad i=1, \ldots,m,
\end{aligned}
\end{equation} 
of dimension $n + \sum_{i=1}^{m}j_{i}$, where we have denoted by $u_{i}^{(j_{i})}$ the $j_{i}$-th order time derivative of $u_{i}$, is locally static feedback linearizable.
%

The corresponding transformation is called \emph{pure dynamic extension} in \cite{Sluis-scl93,SlT-scl96} in the context of  finding upper bounds on the indices $j_{1}, \ldots, j_{m}$, if they exist. We prefer using here the word \emph{prolongation} initially introduced in \cite{CLM}, in reference to \'E. Cartan \cite{C} (see also \cite{Sh}).  

This special class of differentially flat systems, initially introduced in the framework of dynamic feedback linearization  \cite{CLMscl,CLM,Sh,ABMP-ieee}, has also been studied in \cite{JF,FrFo-ejc2005}, under the name of \emph{linearization by prolongation}, again in terms of upper bounds of prolongation orders, and in \cite{BC-scl2004} where an algorithm, yielding a sufficient condition, is obtained. 

However, the necessary and sufficient conditions presented in this paper, tackled via the here introduced concept of \emph{equivalence by pure prolongation}, are original to the author's knowledge. Moreover, they are easily checkable by an algorithm that only involves Lie bracket and linear algebra operations, yielding the minimal pure prolongations if they exist.

Our main contributions are:
\begin{itemize}
\item[1.] The introduction of the \textbf{equivalence relation of systems by pure prolongation}, strictly coarser than the well-known equivalence by diffeomorphism and feedback (see \eg \cite{JR,hunt-et-al-ieee83,I-1sted,NvdS,L-GNL3}) and strictly finer than the Lie-B\"acklund equivalence (see \eg \cite{FLMR-ieee,L_book}). For this equivalence relation, the set of flat systems by pure prolongation is identified with the equivalence class of 0 (modulo the trivial vector field)(see definition~\ref{flatbypp:def});
\item[2.]\textbf{Proposition~\ref{G-k-j:prop}}, extending and specifying results implicitly present in various forms in \cite{CLM,JF,BC-scl2004,FrFo-ejc2005}, where we prove that, whatever the prolongation $\mbf{j} \triangleq (j_{1}, \ldots, j_{m})$, the filtration of prolonged distributions, built on the successive Lie brackets of the prolonged drift with the prolonged control vector fields, is \textbf{decomposable into the direct sum of two filtrations of distributions}. The first one, denoted by $\{\Delta_{k}^{(\mbf{j})}\}_{k\geq 0}$, is, at each point of the prolonged manifold, generated by vector fields belonging to the original tangent bundle, of dimension $n+m$ (in a suitably defined vertical bundle, see \eqref{Deltadef:eq}); The second one, denoted by $\{\Gamma_{k}^{(\mbf{j})}\}_{k\geq 0}$, is generated by the sequence of prolonged control vector fields of decreasing orders up to 1, \ie  $\Gamma_{k}^{(\mbf{j})} = \bigoplus_{i=1, \ldots, m}\{\ddu{i}{j_{i}-r} \mid r=0, \ldots, \min(j_{i}-1, k)\}$ (see \eqref{Gammadef:eq});
\item[3.] \textbf{Theorem~\ref{cns0:thm}} giving the \textbf{necessary and sufficient conditions}: (i) $\Delta_{k}^{(\mbf{j})}$ must be involutive with locally constant dimension and (ii) invariant by $\Gamma_{k}^{(\mbf{j})}$ for all $k$, and (iii)  $\dim \Delta_{k}^{(\mbf{j})}$ must be equal to $n+m$ for all $k$ large enough (strong controllability);
\item[4.] \textbf{Formula~\eqref{sig-max:def} and theorem~\ref{alg-proof:thm}} giving the prolongation lengths, if ithey exist, or otherwise a criterion of non flatness by pure prolongation;
\item[5.] \textbf{Algorithm \ref{algo1}} whose input comprises the system vector fields and whose output is either the list of minimal prolongation lengths or the claim that the system is not flat by pure prolongation, deduced in a finite number of steps.
\end{itemize}

The paper is organized as follows: In section~\ref{recall:sec}, we present the necessary recalls of basic results on differential flatness and feedback linearization. Then we introduce and study the purely prolonged distributions and the associated vertical ones in section~\ref{access:sec}. The equivalence by pure prolongation and the definition of flatness by pure prolongation, followed by its necessary and sufficient conditions and by the pure prolongation algorithm are then presented in section~\ref{purepro:sec}. A series of five examples, four with two inputs (sections \ref{chainex:subsec}, \ref{driftlessex:subsec}, \ref{5:3:ex} and \ref{pend:ex}) and one with three inputs (section \ref{3in:ex}) then illustrate our results in section~\ref{5ex:ex}. Moreover, the pendulum example (section~\ref{pend:ex}), is proven to be non flat by pure prolongation, though known to be differentially flat \cite{FLMR-ieee,L_book}, thus proving that the set of flat systems by pure prolongation is strictly contained in the set of differentially flat ones. The paper ends with concluding remarks and an appendix establishing a comparison formula between prolonged and non prolonged Lie brackets.

\section{Recalls on Differential Flatness and Feedback Linearization}\label{recall:sec}

Consider a non-linear system over a smooth $n$-dimensional manifold $X$ given by
\begin{equation}\label{nlsys:eq}
\dot{x} = f(x,u)
\end{equation}
where $x$ is the $n$-dimensional state vector, $u \in \R^m$ the input or control vector, with $m \leq n$, and $f$ a $C^{\infty}$ vector field in the tangent bundle $\T X$ of $X$ for each $u\in \R^{m}$ and whose dependence on $u$ is of class $C^{\infty}$. 

Before going further, we need the following notations, valid all along this paper:
\begin{itemize}
\item Boldface letters $\mbf{j}, \mbf{k},\ldots,$ are systematically used to denote multi-integers, \ie 
$$\mbf{j} \triangleq (j_{1}, j_{2},\ldots, j_{m}),  \quad \mbf{k} \triangleq (k_{1}, k_{2},\ldots, k_{m}), \quad \ldots$$
\item We denote the minimum of two arbitrary integers $k$ and $l$ by $k\lor l \triangleq \min\{k,l\}$, and their maximum by $k\wedge l \triangleq \max(k,l)$.

Also, for every $\mbf{j} \triangleq \left( j_{1}, \ldots, j_{m} \right)\in \N^{m}$ and $\mbf{k} \triangleq \left( k_{1}, \ldots, k_{m} \right)\in \N^{m}$, we use the \emph{componentwise minimum} notation
\begin{equation}\label{min-jj-kk}\mbf{j} \bigvee \mbf{k} \triangleq \left( j_{1} \vee k_{1}, \ldots, j_{m} \vee k_{m} \right) = \left( \min(j_{1}, k_{1}), \ldots, \min(j_{m}, k_{m}) \right).
\end{equation}
Accordingly, the \emph{componentwise maximum} is denoted by
\begin{equation}\label{max-jj-kk}\mbf{j} \bigwedge \mbf{k} \triangleq \left( j_{1} \wedge k_{1}, \ldots, j_{m} \wedge k_{m} \right) = \left( \max(j_{1}, k_{1}), \ldots, \max(j_{m}, k_{m}) \right).
\end{equation}
If $k_{1}=\cdots = k_{m} = k \in \N$, we also use the notations $\mbf{j} \bigwedge k$ and $\mbf{j} \bigvee k$, with $k$ in regular math typeface.
\item Overlined symbols will denote a collection of successive time derivatives of a time-dependent function as follows. Given a multi-integer $\mbf{k}=(k_1,\ldots,k_m) \in \N^m$ and a locally defined $C^{\infty}$ function $t\mapsto \xi(t)\in \R^m$, 
\begin{itemize}
\item $\ol{\xi}^{(\mbf{k})}$ denotes the vector 
$\left( \xi_{1}, \dot{\xi}_{1}, \ldots, \xi_{1}^{(k_{1})}, \ldots, \xi_m, \dot{\xi}_{m}, \ldots, \xi_{m}^{(k_{m})} \right)$ of dimension  $m+ \vert \mbf{k}\vert$, with $\vert \mbf{k}\vert \triangleq \sum_{i=1}^{m} k_{i}$ and $\xi_{i}^{(j)}\triangleq \frac{d^{j}\xi_{i}}{dt^{j}}$, $j=1,\ldots,k_{i}$, $i=1,\ldots,m$;
\item$\ol{\xi}$ denotes the infinite sequence 
$$\ol{\xi} \triangleq (\xi,\dot{\xi},\ddot{\xi},\ldots) \triangleq \left( \xi_{i}^{(k)} \; ; \; i=1,\dots,m\; ; k\geq 0 \right) \in \R^{m}_{\infty},$$ where $\R^{m}_{\infty} \triangleq \R^{m}\times \R^{m}\times\cdots$ is the product of an infinite number of copies of $\R^m$.
\end{itemize}
\end{itemize}
\subsection{Recalls on Lie-B\"{a}cklund Equivalence}\label{LB-subsec}

For a detailed presentation of the topics of this section, the reader may refer \eg to \cite{FLMR-ieee,Po-banach93,L_book,L-AAECC,Pome09siam}.

\begin{defn}\label{prol:def}
The infinite order jet space prolongation of system \eqref{nlsys:eq} is given by the pair 
$(X\times \R^{m}_{\infty}, C_f)$, where $X\times \R^{m}_{\infty}$ is the product of $X$ with an infinite number of copies of $\R^m$, with coordinates $(x,\ol{u})$, endowed with the Cartan field
\begin{equation}\label{Cf:def}
C_{f}= f(x,u)\frac{\partial}{\partial x} + \sum_{j\geq 0} \sum_{i=1}^{m} u_{i}^{(j+1)}\frac{\partial}{\partial u_{i}^{(j)}},
\end{equation}
defined on $\T X\times \T\R^{m}_{\infty}$, the tangent bundle of $X\times \R^{m}_{\infty}$.
\end{defn}
  
\begin{defn}[Lie-B\"{a}cklund equivalence]\label{LB:def}
Consider two systems:
\begin{equation}\label{sys-equiv:def}
\dot{x}=g(x,u) \quad \mathrm{and} \quad \dot{y}=\gamma(y,v)
\end{equation}
and their respective prolongations $(X\times \R^{m}_{\infty}, C_g)$, with coordinates $(x,\ol{u})$ 
and Cartan field
\begin{equation}\label{sys-equiv-Cg:def}
C_{g}= g(x,u)\frac{\partial}{\partial x} + \sum_{j\geq 0}\sum_{i=1}^{m}
  u_{i}^{(j+1)}\frac{\partial}{\partial u_{i}^{(j)}},
\end{equation}
and $(Y\times \R^{\mu}_{\infty}, C_{\gamma})$, with coordinates $(y,\ol{v})$, 
and Cartan field
\begin{equation}\label{sys-equiv-Cgam:def}
C_{\gamma}= \gamma(y,v)\frac{\partial}{\partial y} + \sum_{j\geq 0}\sum_{i=1}^{\mu} v_{i}^{(j+1)}\frac{\partial}{\partial v_{i}^{(j)}}.
\end{equation}

We say that they are  \emph{Lie-B\"acklund equivalent} at a pair of points $(x_0,\ol{u}_0)$ and $(y_0,\ol{v}_0)$ if, and only if, there exists neighborhoods ${\mathcal N}_{x_{0},\ol{u}_{0}} \subset X\times \R^{m}_{\infty}$ and ${\mathcal N}_{y_{0},\ol{v}_{0}} \subset Y\times \R^{\mu}_{\infty}$ and a $C^{\infty}$ isomorphism\footnote{Recall that a continuous function and, \emph{a fortiori}, differentiable, resp. $C^{\infty}$, depends, by definition of the source and target product topologies, on a finite number of components of its variables, namely $\Phi$ (resp. $\Psi$) depends an a finite number of components of $(y,\ol{v})$ (resp. $(x,\ol{u})$)  (see \eg \cite{KLV,Z,L_book}).} $\Phi : {\mathcal N}_{y_{0},\ol{v}_{0}} \rightarrow {\mathcal N}_{x_{0},\ol{u}_{0}}$ satisfying $\Phi(y_{0},\ol{v}_{0})= (x_0,\ol{u}_0)$, with $C^{\infty}$ inverse  $\Psi$, such that the respective Cartan fields are $\Phi$ and $\Psi$ related, \ie $\Phi_{\ast}C_{\gamma}=C_{g}$ in ${\mathcal N}_{x_{0},\ol{u}_{0}}$ and $\Psi_{\ast}C_{g}=C_{\gamma}$ in ${\mathcal N}_{y_{0},\ol{v}_{0}}$.
\end{defn}

In other words, the two systems are \emph{Lie-B\"acklund equivalent} at the points $(x_0,\ol{u}_0)$ and $(y_0,\ol{v}_0)$  if there exist neighborhoods of these points where every integral curve of the first system is mapped to an integral curve of the second one and conversely, with the same time parameterization. Clearly, this relation is an equivalence relation.

We recall, without proof, a most important result from \cite{pM1} (see also 
\cite{FLMR-ijc,FLMR-ieee,L_book}) giving an interpretation of the Lie-B\"{a}cklund equivalence in terms of diffeomorphism in finite dimension and endogenous dynamic feedback, that will be useful later on. 

\begin{thm}[Martin~\cite{pM1}]\label{endo-equiv:thm}
If the two systems \eqref{sys-equiv:def}
are Lie-B\"{a}cklund equivalent at a given pair of points, then (i) and (ii) must be satisfied:
\begin{itemize}
\item[(i)] $m=\mu$, \ie they must have the same number of independent inputs;
\item[(ii)]
there exist 
\begin{itemize}
\item an endogenous dynamic feedback\footnote{A dynamic feedback is said \emph{endogenous} if, and only if, the closed-loop system and the original one are Lie-B\"{a}cklund equivalent, \ie if, and only if, the extended state $z$ can be locally expressed as a smooth function of $x$, $u$ and a finite number of time derivatives of $u$ (see \cite{pM1,FLMR-ijc,FLMR-ieee,L_book}).} 
\begin{equation}\label{dynfeed:def}
u=\alpha(x,z,w), \quad  \dot{z}=\beta(x,z,w),
 \end{equation}
 where $z$ belongs to $Z$, a finite dimensional smooth manifold,
\item a multi-integer\footnote{Recall that we denote by 
$v^{(\mbf{r})} \triangleq \left( v_{1}^{(r_{1})}, \ldots, v_{m}^{(r_{m})} \right) \triangleq 
\left( \frac{\dd^{r_{1}}v_{1}}{\dd t^{r_{1}}}, \ldots, \frac{\dd^{r_{m}}v_{m}}{\dd t^{r_{m}}}\right)$.} 
$\mbf{r} \triangleq \left( r_{1}, \ldots, r_{m}\right)$,
\item and a local diffeomorphism  $\chi: Y\times \R^{\vert \mbf{r}\vert}  \rightarrow  X\times Z$, 
\end{itemize}
all defined in a neighborhood of the considered points, such that the extended system
\begin{equation}\label{CLdynfeed-equiv:eq}
\dot{y}= \gamma(y,v), \quad v^{(\mbf{r})}=w
\end{equation} 
and the closed-loop one 
\begin{equation}\label{CLdynfeed:eq}
\dot{x}=g(x,\alpha(x,z,w)), \quad \dot{z}=\beta(x,z,w)
\end{equation} 
are $\chi$-related for all $w\in \R^{m}$, \ie 
\begin{equation}\label{extended-diffeo:eq}
(x,z)=\chi(y,v, \dot{v},\ldots,v^{(\mbf{r-1})}), \qquad 
(y,v, \dot{v},\ldots,v^{(\mbf{r-1})})=\chi^{-1}(x,z)
\end{equation}
and 
\begin{equation}\label{extended-diffeo-fields:eq}
\hat{g} =\chi_{\ast}\hat{\gamma}, \qquad
\hat{\gamma} = \chi^{-1}_{\ast}\hat{g}
\end{equation}
where we have denoted
$$
\begin{aligned}
&\hat{g}(x,z,w) \triangleq g(x,\alpha(x,z,w))\frac{\partial}{\partial x} + \beta(x,z,w)\frac{\partial}{\partial z}\\
&\hat{\gamma}(y,v, \dot{v}, \ldots, v^{(\mbf{r-1})},w) \triangleq \gamma(y,v)\frac{\partial}{\partial y} + \sum_{i=1}^{m}\sum_{j=0}^{r_{i}-1} v_{i}^{(j+1)}\frac{\partial}{\partial v_{i}^{(j)}} + w_{i}\frac{\partial}{\partial v_{i}^{(r_{i})}}.
\end{aligned}
$$
\end{itemize}
\end{thm}

\subsection{Recalls on Differential Flatness}\label{flat:subsec}
\begin{defn}\label{flat:def}
We say that system  \eqref{nlsys:eq} is \emph{differentially flat} (or, more shortly, \emph{flat}) at the pair of points $(x_0,\ol{u}_0)$ and $\ol{y}_{0} \in \R^{m}_{\infty}$ if and only if, it is \emph{Lie-B\"{a}cklund equivalent to the trivial system ($\R^{m}_{\infty}, \tau)$ where $\tau$ is the \emph{trivial} Cartan field
\begin{equation}\label{triv:def}
\tau \triangleq \sum_{j\geq 0}\sum_{i=1}^{m} y_{i}^{(j+1)}\frac{\partial}{\partial y_{i}^{(j)}}
\end{equation}
at the considered points}.
\end{defn}

Otherwise stated, the locally defined flat output $y=\Psi(x, \ol{u})$ is such that $(x,\ol{u})= \Phi(\ol{y}) \triangleq (\Phi_{-1}(\ol{y}), \Phi_0(\ol{y}), \Phi_1(\ol{y}), \ldots)$, with\footnote{The first component of $\Phi$, corresponding to the $x$ component, is denoted by $\Phi_{-1}$ so that the component $\Phi_i$ effectively corresponds to the $i$th time derivative of $u$.} $\ddt{\Phi}_{-1}(\ol{y})\equiv f(\Phi_{-1}(\ol{y}), \Phi_0(\ol{y}))$ for all sufficiently differentiable  function $y: t\in \R \mapsto y(t)\in \R^m$.

This definition immediately implies that a system is flat if, and only if, there exists a generalized output $y= \Psi(x,\ol{u})$ of dimension $m$, depending at most on a finite number of derivatives of $u$, with independent derivatives of all orders, such that $x$ and $\ol{u}$ can be expressed in terms of $y$ and a finite number of its successive derivatives, \ie $(x,\ol{u})= \Phi(\ol{y})$, and such that the system equation $\ddt{\Phi_{-1}}(\ol{y})= f \circ \Phi(\ol{y})$ is identically satisfied for all sufficiently differentiable $y: \R \rightarrow \R^m$. 

For a flat system, with the notations of theorem~\ref{endo-equiv:thm}, the vector field $\gamma$, or $\widehat{\gamma}$ indifferently, corresponds  to the linear system in Brunovsk\'{y} canonical form
\begin{equation}\label{flat-CLdynfeed-equiv:eq}
y_{i}^{(r_{i}+1)}=w_{i}, \qquad i=1,\ldots,m,
\end{equation}
and $C_{\gamma}$, defined by \eqref{sys-equiv-Cgam:def}, is the trivial Cartan field $C_{\gamma} = \tau$, with $\tau$ given by \eqref{triv:def}.
\begin{rem}\label{0-LB:rem}
In view of Definition~\ref{LB:def}, since the trivial Cartan field \eqref{triv:def} is the infinite jet prolongation of    $\dot{y}_{i}=u_{i}$, $i=1, \ldots, m$, system \eqref{nlsys:eq} is flat if, and only if, it is Lie-B\"{a}cklund equivalent to $m$ simple integrators in parallel. We may also remark that, since the Cartan field is equal to 0 (modulo $\tau)$, the class of flat systems may be identified with the Lie-B\"{a}cklund equivalence class of 0 (modulo $\tau$).
\end{rem}

Theorem~\ref{endo-equiv:thm} reads:
\begin{cor}
If system \eqref{nlsys:eq} is flat at a given point, there exists an endogenous dynamic feedback of the form \eqref{dynfeed:def}, a multi-integer $\mbf{r} \triangleq \left( r_{1}, \ldots, r_{m}\right)$ and a finite dimensional local diffeomorphism $\chi$ such that systems  \eqref{CLdynfeed:eq}, with $f$ in place of $g$, and \eqref{flat-CLdynfeed-equiv:eq} are $\chi$-related for all $w\in \R^{m}$.
\end{cor}

Consequently, a flat system can be transformed, by diffeomorphism and feedback of a suitably extended space, in a linear controllable system, which motivates the recalls of the next section.

\subsection{Recalls on Feedback Linearization}


Feedback linearizable systems \cite{JR,hunt-et-al-ieee83} (see also \cite{I-1sted,NvdS,M,L-GNL3}) constitute a  subclass of differentially flat systems. They correspond to the equivalence class of linear controllable systems with respect to the following finer equivalence relation, called \emph{equivalence by diffeomorphism and feedback}.

\begin{defn}\label{diff-feed-equiv:def}
The two systems given by \eqref{sys-equiv:def} are said equivalent by diffeomorphism and feedback if, and only if, there exists 
\begin{itemize}
\item a local diffeomorphism $\varphi$ from a neighborhood of an equilibrium point of $X$ (which may be chosen, without loss of generality, as the origin $0\in X$) to a suitable neighborhood of $0\in Y$,
\item and a \emph{static} feedback $v=\alpha(x,u)$, $\alpha$ being invertible with respect to $u$ for all $x$ in the above mentioned neighborhood of the origin, \ie $\rk{\frac{\partial\alpha}{\partial u}} (x,u) = m$ for all $x$ and $u$ as above, 
\end{itemize} 
such that 
$\varphi^{-1}_{\ast}(f(x,u)) = \gamma \circ (\varphi \times \alpha)(x,u) = \gamma(y,v)$,
\end{defn}
Indeed, this equivalence implies that $n = \dim X = \dim Y$ and that both $u$ and $v$ are $m$-dimensional.

\begin{defn}[Feedback Linearizability]\label{feedlin:def}
System \eqref{nlsys:eq} is said \emph{static feedback linearizable} or, shortly, \emph{feedback linearizable} if the context allows, if, and only if, it is equivalent by diffeomorphism and feedback to a linear system in Brunovsk\'y controllability canonical form
\begin{equation}\label{lin-sys:eq}
y_{i}^{(r_{i})}= v_{i}, \quad i=1,\ldots, m,
\end{equation}
where the multi-integer $\mbf{r} \triangleq (r_{1}, \ldots, r_{m})$, whose components $r_i$ are called the \emph{controllability indices} of system~\eqref{lin-sys:eq}, satisfies $\vert \mbf{r} \vert \triangleq \sum_{i=1}^{m}r_{i}=n = \dim X$. 
\end{defn}

Indeed, since the $n$-dimensional vector 
$$\ol{y}^{(\mbf{r-1})} \triangleq (y_{1}, \ldots, y_{1}^{(r_{1}-1)}, \ldots, y_{m}, \ldots, y_{m}^{(r_{m}-1)})$$ 
and $x$ are diffeomorphic (again we have noted $\mbf{r-1} \triangleq (r_{1}-1, \ldots, r_{m}-1)$), it is immediate to verify that $(y_{1}, \ldots, y_{m})$ is a flat output and that every feedback linearizable system is differentially flat.

These systems have been first characterized by \cite{JR, HSM,hunt-et-al-ieee83} for control-affine systems, \ie systems given by $f(x,u)= f_{0}(x) + \sum_{i=1}^{m}u_{i}f_{i}(x)$.
More generally, it can be easily proven that systems of the form \eqref{nlsys:eq} are feedback linearizable if, and only if, the first-order control-affine prolongation
\begin{equation}\label{control-aff-sys:eq}
\begin{aligned}
&\dot{x}=f(x,u), \\
&\dot{u}_{i}= u_{i}^{(1)}, \quad i=1,\ldots,m
\end{aligned}
\end{equation} 
with state $\ol{x}^{(\0)}\triangleq (x,u) = (x_{1}, \ldots, x_{n}, u_{1},\ldots, u_{m}) \in X^{(\0)} \triangleq X\times \R^{m}$ and control vector $u^{(\1)}\triangleq ( u_{1}^{(1)}, \ldots, u_{m}^{(1)} ) \in \R^{m}$, is feedback linearizable (see \eg \cite{CLMscl,vdS-scl84,Sluis-scl93,L-GNL3}).

Indeed, 
in the local coordinates\footnote{We introduce the superscript $^{(\0)}$ from now on to get ready to work with higher order prolongations (see section~\ref{access:sec}).} $\ol{x}^{(\0)}$ of $X^{(\0)}$, denoting the associated vector fields by \begin{equation}\label{prolong1vecf:eq}
g_{0}^{(\0)}(\ol{x}^{(\0)})\triangleq \sum_{i=1}^{n}f_{i}(x,u)\frac{\partial}{\partial x_{i}}, \qquad 
g_{i}^{(0)}(\ol{x}^{(\0)})\triangleq \frac{\partial}{\partial u_{i}}, \quad i=1,\ldots,m,
\end{equation}
defined on the tangent bundle $\T X^{(\0)} = \T X \times \T \R^{m}$, system \eqref{control-aff-sys:eq} reads
\begin{equation}\label{prolong1sys:eq}
\dot{\ol{x}}^{(\0)}= g_{0}^{(\0)}(\ol{x}^{(\0)}) + \sum_{i=1}^{m} u_{i}^{(1)} g_{i}^{(0)}(\ol{x}^{(\0)})
\end{equation} 
with the usual abuse of notations identifying a vector field expressed in local coordinates with its associated (Lie derivative) first order partial differential operator.

Until now, for simplicity's sake, a system \eqref{nlsys:eq} will always be considered in the form \eqref{prolong1sys:eq}, even if the vector-field $f$ is already given in control-affine form. For the sake of coherence, we set $u=(u_1, \ldots, u_m) \triangleq u^{(\0)} = (u_{1}^{(0)}, \ldots, u_{m}^{(0)})$, so that $\dot{u}^{(\0)}=u^{(\1)}$.

\subsection{Recalls on Lie brackets and Distributions}

The Lie bracket $[\eta, \gamma]$ of two arbitrary vector fields $\eta$ and $\gamma$ of $\T X^{(\0)}$ is given, in the $\ol{x}^{(\0)}$-coordinates, by $[\eta, \gamma] \triangleq \sum_{i=1}^{n+m} \sum_{j=1}^{n+m} \left( \eta_{j}\frac{\partial \gamma_{i}}{\partial x_{j}^{(0)}} -  \gamma_{j}\frac{\partial \eta_{i}}{\partial x_{j}^{(0)}} \right)\frac{\partial}{\partial x_{i}^{(0)}}$, with $\ol{x}^{(\0)}  = (x,u^{(\0)}) \triangleq (x_{1}^{(0)}, \ldots, x_{n+m}^{(0)})$.

For iterated Lie brackets, we use the notation $\ad_{\eta}\gamma \triangleq [\eta,\gamma]$ and $\ad_{\eta}^{k}\gamma \triangleq [\eta,\ad_{\eta}^{k-1}\gamma]$ for $k\geq 1$, with the convention that $\ad_{\eta}^{0}\gamma = \gamma$. 
In addition, if $\Gamma$ is an arbitrary distribution of vector fields on $\T X^{(\0)}$, we note $\ad_{\eta}^{k}\Gamma  \triangleq \{ \ad_{\eta}^{k}\gamma: \gamma\in \Gamma\}$.

The distribution $\Gamma$ is said \emph{involutive} if, and only if, $[\eta, \gamma]\in \Gamma$ for every pair of vector fields $\eta, \gamma \in \Gamma$, to which case we note $[\Gamma,\Gamma]\subset \Gamma$, or $\ol{\Gamma}=\Gamma$, where $\ol{\Gamma}$ denotes the involutive closure of $\Gamma$, \ie the smallest involutive distribution containing $\Gamma$.

If the distribution $\Gamma$ is locally generated by $p$ vector fields $\gamma_{1},\ldots,\gamma_{p}$, with $p$ arbitrary, we write $\Gamma \triangleq \{\gamma_{1},\ldots,\gamma_{p}\}$.
We also denote by $\Gamma(\xi) \triangleq \{\gamma_{1}(\xi),\ldots,\gamma_{p}(\xi)\}$ the vector space generated by the vectors $\gamma_{1}(\xi),\ldots,\gamma_{p}(\xi)$ at a point $\xi \in X^{(\0)}$.
 
Consider the ($\0$th-order or non prolonged) filtration of distributions built on the vector fields \eqref{prolong1vecf:eq}\footnote{As before, the superscript $^{(\0)}$ is used to indicate that the distributions $G_{k}^{(\0)}$ and the related indices $\rho_{k}^{(\0)}$ and $\kappa_{k}^{(\0)}$ are  built on the non prolonged vector fields \eqref{prolong1vecf:eq} and to distinguish them from the prolonged distributions $G_{k}^{(\mbf{j})}$ of arbitrary $\mbf{j}$th order, $\mbf{j}\in \N^{m}$, and related indices, $\rho_{k}^{(\mbf{j})}$ and $\kappa_{k}^{(\mbf{j})}$, later introduced in sections~\ref{access:sec} and~\ref{purepro:sec}.}:
\begin{equation}\label{dist0:def}
G_{0}^{(\0)} \triangleq \{g_{1}^{(\0)},\ldots, g_{m}^{(\0)}\}, \qquad G_{k+1}^{(\0)} \triangleq G_{k}^{(\0)} + \ad_{g_{0}^{(\0)}} G_{k}^{(\0)}, \quad \forall k\geq 0,
\end{equation}
indeed satisfying $G_{0}^{(\0)} \subset \cdots \subset G_{k}^{(\0)} \subset G_{k+1}^{(\0)} \subset \cdots \subset \T X^{(\0)}$.

\begin{thm}[\cite{JR,hunt-et-al-ieee83}]\label{feedlin-thm}
System \eqref{nlsys:eq}, or equivalently system \eqref{control-aff-sys:eq},  is feedback linearizable in a neighborhood of the origin of $X^{(\0)}$ if, and only if, in this neighborhood:
\begin{itemize}
\item[(i)] $G_{k}^{(\0)}$ is involutive with constant dimension for all $k \geq 0$,
\item[(ii)] there exists an integer $k^{(\0)}_{\star} \leq n$ such that $G_{k}^{(\0)}=G_{k^{(\0)}_{\star}}^{(\0)}= \T X^{(\0)}$ for all $k\geq k^{(\0)}_{\star}$.
\end{itemize}
\end{thm}

Note that, according to \eqref{prolong1vecf:eq} and \eqref{dist0:def}, $G_{0}^{(\0)} = \{\ddu{1}{0}, \ldots, \ddu{m}{0}\}$ is involutive with constant dimension, equal to $m$, by construction.

Theorem~\ref{feedlin-thm} provides a construction of flat outputs via Frobenius theorem (see \eg \cite{Ch}) and the list of the so-called \emph{Brunovsk\'{y}'s controllability indices}  \cite{Bru} as follows:
\begin{defn}\label{brunov-0:def}
Consider the sequence of integers 
$$\rho_{k}^{(\0)} \triangleq \dim  G_{k}^{(\0)} / G_{k-1}^{(\0)}  \quad \forall k\geq 1,  \qquad \rho_{0}^{(\0)} \triangleq \dim G_{0}^{(\0)} = m.$$
The Brunovsk\'{y} controllability indices $\kappa_{k}^{(\0)}$'s are defined by
$$\kappa_{k}^{(\0)} \triangleq \# \{ l \mid \rho_{l}^{(\0)} \geq k \}, \quad k= 1,\ldots, m,$$
where $\#A$ denotes the number of elements of an arbitrary set $A$.
\end{defn}

It can be proven (see \eg \cite{JR,hunt-et-al-ieee83,I-1sted,NvdS,L-GNL3}) that, for a feedback linearizable nonlinear system \eqref{nlsys:eq}, or \eqref{control-aff-sys:eq}, we have:
\begin{itemize}
\item the sequences $\{\rho_{k}^{(\0)}\}_{k\geq 0}$ and $\{\kappa_{k}^{(\0)}\}_{k\geq 0}$ are non increasing, 
\item $\rho_{k}^{(\0)}\leq m$  for all $k, \quad\rho_{k}^{(\0)}=0$, for all $k\geq k^{(\0)}_{\star}+1$,
\item $\kappa_{1}^{(\0)} = k^{(\0)}_{\star} + 1, \quad\kappa_{m}^{(\0)} \geq 1$, 
\item $\ds \sum_{k=0}^{k^{(\0)}_{\star}} \rho_{k}^{(\0)} = \sum_{k=1}^{m} \kappa_{k}^{(\0)} =  \dim G_{k^{(\0)}_{\star}}^{(\0)}  = n+m$.
\end{itemize}

The list $\kappa_{1}^{(\0)}, \ldots, \kappa_{m}^{(\0)}$ is uniquely defined up to input permutation, invariant by static state feedback and state diffeomorphism,  and is indeed equal to the list of controllability indices of the associated linear system \eqref{lin-sys:eq} with $r_{i}= \kappa_{i}^{(\0)}$, $i=1, \ldots, m$. 

Moreover, for all $k$ and all $i=1,\ldots, m$, and possibly up to a suitable input reordering, we have 
$$G_{k}^{(\0)} = \bigoplus_{j=1}^{m}  \{\ad_{g_{0}^{(\0)}}^{l}g_{j}^{(\0)} \mid l= 0, \ldots, k\lor (\kappa_{j}^{(\0)}-1)\}, \quad G_{\kappa_{1}^{(\0)}-1}^{(\0)} = G_{k^{(\0)}_{\star}}^{(\0)} = \T X^{(\0)}.$$

Then, flat outputs $(y_{1}, \ldots, y_{m})$ are locally non trivial solutions of the system of PDE's
\begin{equation}\label{PDEs0:eq}
\begin{aligned}
&L_{\ad_{g_{0}^{(\0)}}^{k}g_{j}^{(\0)}}\; y_{i} = 0, \; k= 0,\ldots, \kappa_{i}^{(\0)}-2, \; j=1, \ldots, m, \\
&\mathrm{with~}                           
L_{\ad_{g_{0}^{(\0)}}^{\kappa_{i}^{(\0)}-1}g_{i}^{(\0)}} \;y_{i} \neq 0,
\end{aligned}
\end{equation}
for $i=1, \ldots,m$, where we have denoted by $L_{\eta}\varphi$ the Lie derivative of a vector function $\varphi$ along the vector field $\eta$.
These solutions are such that the mapping
$$\ol{x}^{(\0)} \mapsto (y_{1}, \ldots, y_{1}^{(\kappa_{1}^{(\0)}-1)}, \ldots , y_{m}, \ldots, y_{m}^{(\kappa_{m}^{(\0)}-1)})$$
is a local diffeomorphism.

\begin{rem}\label{singleinputequiv:rem}
Recall from \cite{CLMscl} that, for single input systems, differential flatness and feedback linearizability are equivalent. 
\end{rem}

\section{System Pure Prolongation}\label{access:sec}
\subsection{Purely Prolonged Distributions}

We now introduce higher order prolongations of the vector fields defined by \eqref{prolong1vecf:eq}, called \emph{pure prolongations} \cite{CLM} (see also \cite{Sluis-scl93,SlT-scl96,BC-scl2004,FrFo-ejc2005}). 


Given a multi-integer $\mbf{j}\triangleq (j_{1},\ldots, j_{m}) \in \N^{m}$, we note, as before, $\vert \mbf{j} \vert \triangleq \sum_{i=1}^{m} j_{i}$ and the prolonged state:
$$\ol{x}^{(\mbf{j})} \triangleq (x, \ol{u}^{(\mbf{j})}) \triangleq (x_{1}, \ldots, x_{n}, u_{1}^{(0)}, \ldots, u_{1}^{(j_{1})}, \ldots, u_{m}^{(0)}, \ldots, u_{m}^{(j_{m})}),$$ 
with $u_{i}^{(0)} \triangleq u_{i}$, $i=1, \ldots, m$.

Let $X^{(\mbf{j})} \triangleq X\times \R^{m+\vert \mbf{j}\vert}$ be the associated $\mbf{j}$-th order jet manifold  of dimension $n + m+ \vert \mbf{j} \vert$ with coordinates $\ol{x}^{(\mbf{\mbf{j}})}$. 

The pure prolongation of order $\mbf{j}$ of system \eqref{control-aff-sys:eq}, or otherwise said, of the vector fields \eqref{prolong1vecf:eq}, in the tangent bundle $\T X^{(\mbf{j})} = \T X\times \T\R^{m+\vert\mbf{j} \vert}$, is defined by
\begin{equation}\label{f_0-j:def}
\begin{aligned}
g_{0}^{(\mbf{j})}(\ol{x}^{(\mbf{j})})&= \sum_{i=1}^{n} f_{i}(x,u) \frac{\partial}{\partial x_{i}} + \sum_{i=1}^{m}\sum_{k=0}^{j_{i}-1} u_{i}^{(k+1)} \frac{\partial}{\partial u_{i}^{(k)}}\\
g_{i}^{(\mbf{j})}(\ol{x}^{(\mbf{j})})&\triangleq g_{i}^{(j_{i})}(\ol{x}^{(\mbf{j})}) \triangleq  \frac{\partial}{\partial u_{i}^{(j_{i})}}, \quad i=1,\ldots,m,
\end{aligned}
\end{equation}
with the convention that $\sum_{k=0}^{j_{i}-1} u_{i}^{(k+1)} \frac{\partial}{\partial u_{i}^{(k)}} \triangleq 0$ if $j_{i} =0$.

They are naturally associated to the adjunction of $\mbf{j}$ pure integrators to $u=u^{(0)}$ in \eqref{control-aff-sys:eq} (with the same usual abuse of notations as in \eqref{prolong1sys:eq}):
\begin{equation}\label{nlsys(j):def}
\dot{\ol{x}}^{(\mbf{j})} = g_{0}^{(\mbf{j})}(\ol{x}^{(\mbf{j})}) + \sum_{i=1}^{m} u_{i}^{(j_{i}+1)} g_{i}^{(j_{i})}(\ol{x}^{(\mbf{j})})
\end{equation} 
or $$\dot{x}=f(x,u), \qquad \dot{u}_{i}^{(k)} = u_{i}^{(k+1)}, \quad k= 0,\ldots, j_{i}, ~~i=1,\ldots, m,$$
$u^{(\mbf{j+1})} = ( u_{1}^{(j_{1}+1)}, \ldots, u_{m}^{(j_{m}+1)})$ being the new control vector of this \emph{purely prolonged system}, whose state is $( x, \ol{u}^{(\mbf{j})}) = \ol{x}^{(\mbf{j})}$.

\begin{rem}
Note that the state of the $\mbf{j}$-th prolonged system, $\ol{x}^{(\mbf{j})}$, coincides with the image of $\ol{x}$ by the projection $p_{\mbf{j}} : \ol{x} \in X\times \R^{m}_{\infty} \mapsto p_{\mbf{j}}(\ol{x}) = \ol{x}^{(\mbf{j})} \in X^{(\mbf{j})}$ for all $\mbf{j}$. 
In addition, the family of projections 
$p_{\mbf{i},\mbf{j}} : \ol{x}^{(\mbf{i})}\in X^{(\mbf{i})} \mapsto \ol{x}^{(\mbf{j})}=
p_{\mbf{i},\mbf{j}}(\ol{x}^{(\mbf{i})}) \in X^{(\mbf{j})}$ for all $\mbf{i}, \mbf{j}$ such that $i_{k}\geq j_{k}$ for all $k=1,\ldots,m$, that we note $\mbf{i}\succeq \mbf{j}$, indeed satisfies $p_{\mbf{i},\mbf{j}} \circ  p_{\mbf{j},\mbf{k}} = p_{\mbf{i},\mbf{k}}$ for all $\mbf{i} \succeq \mbf{j} \succeq \mbf{k}$ and thus allows us to identify the manifold $X\times \R^{m}_{\infty}$ with the \emph{projective limit} of the family $(X^{(\mbf{i})}, p_{\mbf{i},\mbf{j}})$ for all $\mbf{i}$ and all $\mbf{j}$ such that $\mbf{i} \succeq \mbf{j}$, \ie $X\times \R^{m}_{\infty} \simeq \ds\lim_{\leftarrow}X^{(\mbf{i})}$ (see \eg \cite[Chap. I,\S 10]{Bbk}). A similar identification trivially holds for the associated tangent bundles, \ie 
$\T X\times \T\R^{m}_{\infty} \simeq \ds\lim_{\leftarrow}\T X^{(\mbf{i})}$ relatively to the family $\T p_{\mbf{i},\mbf{j}}$ of tangent projections, hence the identification of the Cartan field $C_f$, defined by \eqref{Cf:def}, with $\ds \lim_{\leftarrow} g_{0}^{\mbf{(j)}}$, the projective limit of the vector fields $g_{0}^{\mbf{(j)}}$. 
Nevertheless, this property does not hold for the control vector fields $g_{i}^{\mbf{(j)}}$ since $\T p_{\mbf{j},\mbf{k}}( g_{i}^{\mbf{(j)}})$ is not equal to $g_{i}^{\mbf{(k)}}$, for $i=1, \ldots, m$ and $\mbf{j} \succeq \mbf{k}$. Moreover, the Lie bracket of vector fields is not preserved by this family of projections. This is one of the reasons why prolongations may enlarge the range of the system transformations.
\end{rem}

\begin{rem}
Given an arbitrary point $\ol{x}_{0} \triangleq (x_{0},\ol{u}_{0})$ around which system \eqref{control-aff-sys:eq} is defined, it is convenient to consider the shift 
$$\theta : (x,\ol{u})\in X\times \R^{m}_{\infty}  \mapsto \theta(x,\ol{u}) = (x-x_{0}, \ol{(u-u_{0})}) \triangleq (z,\ol{v})\in X\times \R^{m}_{\infty}$$ 
such that $\ol{x}_{0}$ is mapped to the origin of $\T X \times \T\R^{m}_{\infty}$, denoted by $\ol{0}$, thus inducing the shift of vector fields: 
\begin{equation}\label{theta:def}
\theta_{\star}(g_{i}^{(\mbf{j})})(z,\ol{v}) \triangleq g_{i}^{(\mbf{j})}(z+x_{0},\ol{(v + u_{0})}), \quad i=0, \ldots m,
\end{equation} 
now defined in a neighborhood of $\ol{0}$.
For the sake of simplicity, we will only consider such shifted vector fields in the sequel while keeping the same notation $g_{i}^{(\mbf{j})}$ as before, though abusive, but yet unambiguous.
\end{rem}

We now introduce the following filtration of $\mbf{j}$-th order purely prolonged distributions of $\T X^{(\mbf{j})}$:
\begin{equation}\label{distj:def}
G_{0}^{(\mbf{j})} \triangleq \{\ddu{1}{j_{1}}, \ldots, \ddu{m}{j_{m}}\}, \qquad G_{k+1}^{(\mbf{j})} \triangleq G_{k}^{(\mbf{j})} + \ad_{g_{0}^{(\mbf{j})}} G_{k}^{(\mbf{j})}, \quad \forall k\geq 0
\end{equation}

Indeed, for $\mbf{j} = \mbf{0}$, \ie $j_1= \cdots = j_m=0$, this filtration coincides with the $\0$th order one given by \eqref{dist0:def}.
Similarly to the $\0$th order case, $G_{0}^{(\bj)}$ is involutive with constant dimension, equal to $m$, by construction. 

Moreover, since every $G_{k}^{(\mbf{j})}\subset \T X^{(\mbf{j})}$, with $\dim\T X^{(\mbf{j})} = n + m + \mid \bj \mid$, we have
\begin{prop}\label{kstar-j:prop}
There exists a finite integer $k^{(\mbf{j})}_{\star}$ such that  $G_{k}^{(\mbf{j})} = G_{k^{(\mbf{j})}_{\star}}^{(\mbf{j})}$ for all $k\geq k^{(\mbf{j})}_{\star}$ and
\begin{equation}\label{ineq:kstar-j}
k^{(\bj)}_{\star} \leq n+\mid \bj \mid.
\end{equation}
\end{prop}
\begin{proof}
Since 
$$n+m+ \mid \bj \mid \geq \dim G_{k^{(\bj)}_{\star}}^{(\bj)} = \sum_{k=1}^{k^{(\bj)}_{\star}} \dim G_{k}^{(\bj)} /  G_{k-1}^{(\bj)} + \dim G_{0}^{(\bj)} \geq k^{(\bj)}_{\star} + m,$$ we get \eqref{ineq:kstar-j}.
\end{proof}

\begin{rem}
In full generality, $k^{(\bj)}_{\star}$ depends on the point where it is evaluated. However, if $\dim G_{k}^{(\bj)}$ is constant in an open dense subset of $X^{(\mbf{j})}$ for all large enough $k$, so is $k^{(\bj)}_{\star}$.
\end{rem}


Let us inductively define the $n$-dimensional vector functions $\gamma_{k,i}^{(\mbf{j})}$,
for $k\geq 1$, $i=1,\ldots,m$, and arbitrary $\bj = (j_{1}, \ldots, j_{m})$ as follows:
\begin{equation}\label{gamm-k-j:eq}
\begin{aligned}
\gamma_{k+1,i}^{(\mbf{j})} &\triangleq L_{g_{0}^{(\mbf{j})}}\gamma_{k,i}^{(\mbf{j})}
-\gamma_{k,i}^{(\mbf{j})} \frac{\partial f}{\partial x}
=L_{g_{0}^{(\0)}}\gamma_{k,i}^{(\mbf{j})} + \sum_{p=1}^{m}\sum_{l=0}^{j_{p}-1} u_{p}^{(l+1)}\frac{\partial \gamma_{k,i}^{(\mbf{j})}}{\partial u_{p}^{(l)}}
- \gamma_{k,i}^{(\mbf{j})} \frac{\partial f}{\partial x}
\end{aligned}
\end{equation}
with
\begin{equation}\label{gamm-j-j:eq}
\gamma_{1,i}^{(\mbf{j})}= (-1)^{(j_{i}+1)} \frac{\partial f}{\partial u_{i}^{(0)}}.
\end{equation} 

For an arbitrary $\mbf{j}$ and given $i=1,\ldots,m$, thanks to \eqref{gamm-j-j:eq}, it is readily seen that $\gamma_{1,i}^{(\mbf{j})}$ depends at most of $\ol{x}^{(\0)}$ and thus, if $k \leq j_{i}-1$, thanks to \eqref{gamm-k-j:eq},
$\gamma_{k+1,i}^{(\mbf{j})}$ depends at most of $\ol{x}^{(\mbf{j}\bigvee k)}$.

\bigskip
\subsection{Vertical Distributions of Purely Prolonged Ones}

Before stating the next Lemma, we need to recall the definition of \emph{vertical bundle}.  Given an arbitrary $\mbf{r} \in \N^{m}$ and the fiber bundle $\pi_{\mbf{r}}: X^{(\mbf{r})} \rightarrow \R^{m+\mid\mbf{r}\mid}$, with $\pi_{\mbf{r}}(\ol{x}^{(\mbf{r})})= \ol{u}^{(\mbf{r})}$, its vertical space at $\ol{x}^{(\mbf{r})}$, denoted by $\V_{\ol{x}^{(\mbf{r})}} X^{(\mbf{r})}$,  is the tangent space $\T_{x} X$. Its vertical bundle, denoted by $\V X^{(\mbf{r})}$, is the vector bundle made of the vertical spaces at each $\ol{x}^{(\mbf{r})}$,  \ie the set of linear combinations $\sum_{i=1}^{n} \alpha_{i}(\ol{x}^{(\mbf{r})}) \frac{\partial}{\partial x_{i}}$ whose coefficients $\alpha_{i}$ are smooth functions that depend at most on  $\ol{x}^{(\mbf{r})}$ and where $(x_{1}, \ldots, x_{n})$ are local coordinates of $X$.

The same definition indeed holds for the vertical bundle $\V (X\times \R^{m}_{\infty})$ associated to the fiber bundle $\pi: X\times \R^{m}_{\infty} \rightarrow \R^{m}_{\infty}$, \ie the set of linear combinations of $\frac{\partial}{\partial x_{1}}, \ldots, \frac{\partial}{\partial x_{n}}$ whose coefficients are smooth functions of $\ol{x}$.

We now establish some comparison \textsl{formulae} between Lie brackets of the vector fields of the purely prolonged system and those of the original (non prolonged) one. More complete \textsl{formulae} may be found in Lemma~\ref{prepa1:lem} of the Appendix~\ref{annex:compar:sec}.

\begin{lem}\label{prepa2:lem}
For all $\mbf{j}= \left( j_{1}, \ldots,j_{m}\right)\in \N^{m}$ satisfying $0\leq  j_{1} \leq \ldots \leq j_{m}$, for all $k \leq j_{i}$ and $i=1,\ldots,m$,  we have:
\begin{equation}\label{adkCf0-gij-:eq}
\ad_{g_{0}^{(\mbf{j})}}^{k} \ddu{i}{j_{i}} = 
(-1)^k  \frac{\partial}{\partial u_{i}^{(j_{i}-k)}} 
\end{equation}
and for all $k\geq 1 $:
\begin{equation}\label{ad-gamm-k-j:eq}
\ad_{g_{0}^{(\mbf{j})}}^{j_{i}+k} \ddu{i}{j_{i}} = (-1)^{j_{i}}  \ad_{g_{0}^{(\bj)}}^{k} \ddu{i}{0} =
\gamma_{k,i}^{(\mbf{j})}\frac{\partial}{\partial x} \in \V X^{(\mbf{j}\bigvee (k-1))}, 
\end{equation}
Moreover, we have
\begin{equation}\label{0brack:eq}
\left[ \ddu{p}{j_{p}-k}, \ad_{g_{0}^{(\mbf{j})}}^{l-j_{q}}\ddu{q}{0} \right] = 0, \quad \forall k < j_{p}, \;\; \forall l\geq j_{q} ~\mathrm{s.t.~} k+l < j_{p} + j_{q} +1.
\end{equation}
\end{lem}

\begin{proof}
It is immediately seen that
$$
 \ad_{g_{0}^{(\mbf{j})}} \ddu{i}{j_{i}} = \left[ f \frac{\partial}{\partial x} + \sum_{k=1}^{m}\sum_{l\geq 0} u_{k}^{(l+1)} \frac{\partial}{\partial u_{k}^{(l)}} , \frac{\partial}{\partial u_{i}^{(j_{i})}}\right] 
 = - \frac{\partial}{\partial u_{i}^{(j_{i-1})}}.
$$ 
Iterating this computation up to $k= j_{i}$ yields \eqref{adkCf0-gij-:eq}:
$$
\ad^{j_{i}}_{g_{0}^{(\mbf{j})}} \ddu{i}{j_{i}} = (-1)^{j_{i}} \frac{\partial}{\partial u_{i}^{(0)}} 
.$$

Then, for $k=j_{i}+1$, using the fact that $[\frac{\partial}{\partial u_{k}^{(l)}} , \frac{\partial}{\partial u_{i}^{(0)}}] =0$ for all $i$, $k$ and $l\geq 0$, we have:
\begin{equation}\label{gam-kij:eq}
\begin{aligned}
\ad_{g_{0}^{(\mbf{j})}}^{j_{i}+1} \ddu{i}{j_{i}} &= (-1)^{j_{i}} \left[ f \frac{\partial}{\partial x} + \sum_{k=1}^{m}\sum_{l=0}^{j_{k}-1} u_{k}^{(l+1)} \frac{\partial}{\partial u_{k}^{(l)}} , \frac{\partial}{\partial u_{i}^{(0)}}\right]\\
&= (-1)^{j_{i}} \ad_{g_{0}^{(\mbf{0})}} \ddu{i}{0} =
(-1)^{(j_{i}+1)} \frac{\partial f}{\partial u_{i}^{(0)}}(\ol{x}^{(\0)})\frac{\partial}{\partial x} = \gamma_{1,i}^{(\mbf{j})}\frac{\partial}{\partial x},
\end{aligned}
\end{equation}
which proves that $\ad_{g_{0}^{(\mbf{j})}}^{j_{i}+1} \ddu{i}{j_{i}} \in \V X^{(\0)}$ and that \eqref{ad-gamm-k-j:eq} holds at the order $k=1$.

 Assuming that \eqref{ad-gamm-k-j:eq} holds up to $k=\nu$, with $\gamma_{\nu,i}^{(\mbf{j})}$ depending at most on $\ol{x}^{(\mbf{j}\lor (\nu -1))}$, we have
$$\ad_{g_{0}^{(\mbf{j})}}^{j_{i}+\nu+1} \ddu{i}{j_{i}} = (-1)^{j_{i}}  \left[g_{0}^{(\mbf{j})}, \ad_{g_{0}^{(\mbf{j})}}^{\nu} \ddu{i}{0}\right]= 
\left(L_{g_{0}^{(\mbf{j})}}\gamma_{\nu,i}^{(\mbf{j})} - \gamma_{\nu,i}^{(\mbf{j})}\frac{\partial f}{\partial x}\right) \frac{\partial}{\partial x} = \gamma_{\nu+1,i}^{(\mbf{j})} \frac{\partial}{\partial x},
$$
and, according to \eqref{gamm-k-j:eq}-\eqref{gamm-j-j:eq}, for all $j_{p}+1\leq \nu < j_{p+1}$, $p=1,\ldots,m-1$, since we differentiate $\gamma_{\nu,i}^{(\mbf{j})}$ with respect to $\sum_{k=1}^{m}\sum_{l=0}^{j_{k}-1} u_{k}^{(l+1)} \frac{\partial}{\partial u_{k}^{(l)}}$, it is immediate to verify that $\gamma_{\nu+1,i}^{(\mbf{j})}$ depends at most on 
$$\ol{x}^{(\mbf{j}\bigvee \nu)}= \left(x, u_{1}^{(0)},\ldots, u_{1}^{(j_{1})}, \ldots, u_{p}^{(0)}, \ldots, u_{p}^{(j_{p})}, u_{p+1}^{(0)}, \ldots, u_{p+1}^{(\nu)} , \ldots, u_{m}^{(0)}, \ldots, u_{m}^{(\nu)}  \right)
$$ 
and on $\ol{x}^{(\mbf{j})} = \ol{x}^{(\mbf{j}\bigvee \nu)}$ if $\nu \geq j_{m}$, hence \eqref{ad-gamm-k-j:eq}.

Concerning \eqref{0brack:eq},
since $\ad_{g_{0}^{(\mbf{j})}}^{l-j_{q}}\ddu{q}{0} = \pm \gamma_{l-j_{q},q}^{(\mbf{j})} \ddx$ depends at most on $\ol{x}^{(\mbf{j}\bigvee (l-j_{q}-1))}$, the derivative of  $\gamma_{l-j_{q},q}^{(\mbf{j})}$ with respect to $\ddu{p}{j_{p}-k}$ is indeed null if $j_{p}-k > l-j_{q}-1$, hence the result.
\end{proof}

\subsection{Decomposition of Purely Prolonged Distributions}

Let us assume, without loss of generality, that the control components have been reordered in such a way that $j_{1}\leq j_{2} \leq \cdots \leq j_{m}$. Moreover, we may suppose that $j_{1}=0$, as shown to be sufficient in the next section.

We now introduce two new filtrations of $\T X^{(\bj)}$, noted $\Gamma_{k}^{(\mbf{j})}$ and $\Delta_{k}^{(\mbf{j})}$, for $k\geq 0$, as follows
\begin{equation}\label{Gammadef:eq}
\Gamma_{k}^{(\mbf{j})} 
\triangleq  \bigoplus_{p=2}^{m}\left\{\ddu{p}{j_{p}-l} \mid l= 0, \ldots, k\vee(j_{p}-1) \right\} ,
\end{equation}
\begin{equation}\label{Deltadef:eq}
\Delta_{k}^{(\mbf{j})} \triangleq \sum_{p=1}^{m} \left\{ \ad_{g_{0}^{(\mbf{j})}}^{l-j_{p}}\ddu{p}{0}  \mid l= j_{p}, \ldots, k \right\}.
\end{equation}
with the convention that $\ad_{g_{0}^{(\mbf{j})}}^{k-j_{p}}\ddu{p}{0} = 0$ if $k< j_{p}$, $p=1, \ldots, m$.

We indeed have $\Gamma_{k}^{(\mbf{j})} = \Gamma_{j_{m}-1}^{(\mbf{j})}= \bigoplus_{p=2}^{m}\{\ddu{p}{j_{p}}, \ldots, \ddu{p}{1}\}$ 
for all $k\geq j_{m}-1$. Thus,  $\dim \Gamma_{k}^{(\mbf{j})} \leq \vert \mbf{j}\vert$ for all $k$.

The definitions \eqref{Gammadef:eq}--\eqref{Deltadef:eq} and Lemma~\ref{prepa2:lem} readily yield, for all $k$, $l$ and $\mbf{j}$,
\begin{equation}\label{Gammaprop:eq}
\begin{aligned}
&\ol{\Gamma_{k}^{(\mbf{j})}}= \Gamma_{k}^{(\mbf{j})}, \qquad
[\Gamma_{k}^{(\mbf{j})}, \Gamma_{l}^{(\mbf{j})}] \subset \Gamma_{k\wedge l}^{(\mbf{j})}, \\
&\Delta_{l}^{(\mbf{j})} \subset \V (X\times \R^{m}_{\infty}), \quad
[\Gamma_{k}^{(\mbf{j})}, \Delta_{l}^{(\mbf{j})}]\subset \V (X\times \R^{m}_{\infty}), \quad [\Gamma_{k}^{(\mbf{j})}, \Delta_{l}^{(\mbf{j})}]\cap \Gamma_{k}^{(\mbf{j})} = \{ 0 \}.
\end{aligned}
\end{equation} 

\begin{rem}\label{nondecrease:rem}
Note that, unlike the filtration $\{G_{k}^{(\mbf{j})}\}_{0\leq k \leq k^{(\mbf{j})}_{\star}}$ that is increasing for $k\leq k_{\star}^{(\mbf{j})}$, the mapping $k\mapsto \dim \Delta_{k}^{(\mbf{j})}$ is only non-decreasing in general.
\end{rem}
\begin{rem}
In our definition of $\Delta_{k}^{(\mbf{j})}$, we consider Lie brackets of the form $\ad_{g_{0}^{(\mbf{j})}}^{k}\ddu{p}{0}$, as opposed to \cite{BC-scl2004} where Lie brackets of the form $\ad_{g_{0}^{(\0)}}^{k}\ddu{p}{0}$ are used.
\end{rem}

\begin{prop}\label{G-k-j:prop}
For all $\mbf{j}$ such that  $0= j_{1}\leq \ldots \leq j_{m-1} \leq j_{m}$, with $j_{m}$ finite, if $\mbf{j}= (j_{1}, \ldots, j_{i},j_{i+1}, \ldots, j_{m})$ and $\mbf{j'}= (j_{1}, \ldots, j_{i},j'_{i+1}, \ldots, j'_{m})$ for some $(j'_{i+1}, \ldots, j'_{m})$, we have $\Delta_{k}^{(\mbf{j})} =\Delta_{k}^{(\mbf{j'})}$ for all $k = 0, \ldots, (j_{i+1}\vee j'_{i+1}) -1$.

Moreover, for all $k \geq 0$, 
\begin{equation}\label{dimGamma:eq}
\dim \Gamma_{k}^{(\mbf{j})} = \left\{ 
\begin{array}{ll}
\ds \sum_{p=1}^{i} j_{p} + (k+1)(m-i)& \ds \mathrm{if~} j_{i}\leq k < j_{i+1}, \; i=1,\ldots, m-1\\
\ds \mid \mbf{j} \mid & \ds \mathrm{if~} k\geq j_{m}
\end{array}\right.
\end{equation}
\begin{equation}\label{dimDelta:eq}
\dim \Delta_{k}^{(\mbf{j})} \leq \left\{ 
\begin{array}{ll}
\ds ( (k+1)i - \sum_{p=1}^{i} j_{p} ) \vee (n+i) & \ds \mathrm{if~} j_{i}\leq k < j_{i+1}, \; i=1,\ldots, m-1\\
\ds \left( (k+1)m -\mid \mbf{j} \mid \right) \vee (n+m) & \ds \mathrm{if~} k\geq j_{m}
\end{array}\right.
\end{equation}
and we have
\begin{equation}\label{G-k-j:eq}
G_{k}^{(\mbf{j})} =  \Gamma_{k}^{(\mbf{j})} \oplus \Delta_{k}^{(\mbf{j})}, \quad \forall k\geq 0.
\end{equation}
Furthermore, the finite integer $k_{\star}^{(\mbf{j})}$, satisfying \eqref{ineq:kstar-j}, is such that $\Delta_{k}^{(\mbf{j})} = 
\Delta_{k_{\star}^{(\mbf{j})}}^{(\mbf{j})}$ and $\Gamma_{k}^{(\mbf{j})} = 
\Gamma_{k_{\star}^{(\mbf{j})}}^{(\mbf{j})}$ for all $k\geq k_{\star}^{(\mbf{j})}$.

If, in addition, $\dim \Delta_{k_{\star}^{(\mbf{j})}}^{(\mbf{j})} = m+n$, then
\begin{equation}\label{ineq-j:eq}
n+\mid \bj \mid \geq k_{\star}^{(\mbf{j})} \geq  j_{m}\wedge \frac{n + \mid \mbf{j}\mid}{m}.
\end{equation}
\end{prop}

\begin{proof} 
By definition, the generators of $\Gamma_{k}^{(\mbf{j})}$ are independent for all $k$ and thus their number is equal to $\dim \Gamma_{k}^{(\mbf{j})}$, hence \eqref{dimGamma:eq}. The dimension of $\Delta_{k}^{(\mbf{j})}$, in turn, is lesser than, or equal to, the number of its generators, in number $(k+1) + \ldots + (k-j_{i}+1) = (k+1)i - \sum_{p=1}^{i}j_{p}$,
if $j_{i}\leq k < j_{i+1}$ (respectively $(k+1) + \ldots + (k-j_{m}+1) = (k+1)m - \mid\mbf{j}\mid$, if $k\geq j_{m}$), 
and, since, according to \eqref{ad-gamm-k-j:eq} of Lemma~\ref{prepa2:lem}, $\Delta_{k}^{(\mbf{j})}$ is contained in $\{ \ddu{1}{0},\ldots, \ddu{i}{0}, \ddxx{1},\ldots, \ddxx{n}  \}$ if $j_{i}\leq k < j_{i+1}$, $i=1,\ldots, m-1$ (respectively in $\{ \ddu{1}{0},\ldots, \ddu{m}{0}, \ddxx{1},\ldots, \ddxx{n}  \}$ if $k\geq j_{m}$), its dimension is bounded above by $i+n$ (resp. $m+n$), hence \eqref{dimDelta:eq} 

The proof of \eqref{G-k-j:eq} is by induction. For $k=0$, by \eqref{Gammadef:eq}-\eqref{Deltadef:eq}, we indeed have $G_{0}^{(\mbf{j})}= \left\{\ddu{1}{0}, \ddu{2}{j_{2}}, \ldots, \ddu{m}{j_{m}}\right\} =  \left\{\ddu{1}{0}\right\} \oplus \left\{\ddu{2}{j_{2}}, \ldots, \ddu{m}{j_{m}}\right\}  = \Delta_{0}^{(\mbf{j})}  \oplus \Gamma_{0}^{(\mbf{j})}$. Thus, \eqref{G-k-j:eq} is valid at the order 0. 

Assume now that \eqref{G-k-j:eq} holds true up to the order $\nu>0$ with $j_{r} \leq \nu < j_{r+1}$ for some $r\in \{1, \ldots, m\}$, assuming that  $j_{r} < j_{r+1}$. If $j_{r}= j_{r+1}$, the reader may immediately go to the case $\nu +1 = j_{r+1}$ below.

At the order $\nu + 1$, two cases are possible: either $j_{r} \leq \nu+1 < j_{r+1}$ or $\nu +1 = j_{r+1}$. 
In the first case, using Lemma~\ref{prepa2:lem},  we get:
$$\begin{aligned}
G_{\nu+1}^{(\mbf{j})} &= G_{\nu}^{(\mbf{j})} + ad_{g_{0}^{(\mbf{j})}}G_{\nu}^{(\mbf{j})}  = \Gamma_{\nu}^{(\mbf{j})} \oplus \Delta_{\nu}^{(\mbf{j})} +   ad_{g_{0}^{(\mbf{j})}}\Gamma_{\nu}^{(\mbf{j})} +  ad_{g_{0}^{(\mbf{j})}}\Delta_{\nu}^{(\mbf{j})}\\
&= \Gamma_{\nu}^{(\mbf{j})} \oplus \Delta_{\nu}^{(\mbf{j})} + \left\{ad_{g_{0}^{(\mbf{j})}}^{\nu+1}\ddu{1}{0}, \ldots, ad_{g_{0}^{(\mbf{j})}}^{\nu+1-j_{r}}\ddu{r}{0}, \ddu{r+1}{j_{r+1}-\nu-1}, \ldots, \ddu{m}{j_{m}-\nu-1}\right\}\\
&= \Gamma_{\nu+1}^{(\mbf{j})} \oplus \Delta_{\nu+1}^{(\mbf{j})}.
\end{aligned}$$
In the second case, namely if $\nu +1 = j_{r+1}$, 
$$\begin{aligned}
G_{\nu+1}^{(\mbf{j})} &= \Gamma_{\nu}^{(\mbf{j})} \oplus \Delta_{\nu}^{(\mbf{j})} 
\\
& \hspace{1cm}+ \left\{ad_{g_{0}^{(\mbf{j})}}^{\nu+1}\ddu{1}{0}, \ldots, ad_{g_{0}^{(\mbf{j})}}^{\nu+1-j_{r}}\ddu{r}{0}, \ddu{r+1}{0}, \ddu{r+2}{j_{r+2}-\nu-1}, \ldots, \ddu{m}{j_{m}-\nu-1}\right\}\\
&= \Gamma_{\nu+1}^{(\mbf{j})} \oplus \Delta_{\nu+1}^{(\mbf{j})}.
\end{aligned}$$
The case $j_{m} \leq \nu+1$ follows the same lines:
$$G_{\nu+1}^{(\mbf{j})} 
=  \Gamma_{\nu}^{(\mbf{j})} \oplus \Delta_{\nu}^{(\mbf{j})} + \left\{ad_{g_{0}^{(\mbf{j})}}^{\nu+1}\ddu{1}{0}, \ldots, ad_{g_{0}^{(\mbf{j})}}^{\nu+1-j_{m}}\ddu{m}{0}\right\}
=  \Gamma_{\nu+1}^{(\mbf{j})} \oplus \Delta_{\nu+1}^{(\mbf{j})}
$$
hence \eqref{G-k-j:eq} is proven and the property of the number of iterations $k_{\star}^{(\mbf{j})}$ to simultaneously saturate the dimensions of $\Gamma_{k}^{(\mbf{j})}$ and $\Delta_{k}^{(\mbf{j})}$ immediately follows.

Moreover, if $\dim \Delta_{k_{\star}^{(\mbf{j})}}^{(\mbf{j})} = m+n$, we must have $j_{m} \leq k_{\star}^{(\mbf{j})}$ since otherwise, using definition \eqref{Deltadef:eq} for $k_{\star}^{(\mbf{j})} < j_{m}$,  $\Delta_{k_{\star}^{(\mbf{j})}}^{(\mbf{j})}$ would not contain $\{\ddu{m}{0}\}$ and its dimension would not exceed $m-1+n$. Consequently, applying once more \eqref{dimDelta:eq} with  $\dim \Delta_{k_{\star}^{(\mbf{j})}}^{(\mbf{j})} = m+n$, we get $(k_{\star}^{(\mbf{j})}+1)m -\mid \mbf{j} \mid \geq m+n$, which, combined with \eqref{ineq:kstar-j}, immediately yields \eqref{ineq-j:eq}.
\end{proof}

\begin{rem}
The inequality \eqref{ineq-j:eq} reads $k_{\star}^{(\mbf{j})} -\frac{\mid \mbf{j} \mid}{m} \geq \frac{n}{m}$ and may thus be interpreted as an estimate of the gap between $k_{\star}^{(\mbf{j})}$ and the average value $\frac{\mid \mbf{j} \mid}{m}$ of the prolongation lengths $j_{1}, j_{2}, \ldots, j_{m}$, provided that the prolonged system satisfies the strong accessibility rank condition $\dim \Delta_{k_{\star}^{(\mbf{j})}}^{(\mbf{j})} = m+n$.
\end{rem}

\section{Flatness by Pure Prolongation}\label{purepro:sec}

\subsection{Equivalence by Pure Prolongation}
Consider the two systems \eqref{sys-equiv:def} with $\dim x= n$, $\dim y= n'$ and $\dim u = \dim v = m$. 

Given arbitrary $\mbf{j}$ and $\mbf{l}$, we recall that the associated prolonged vector fields are
$$
g^{(\mbf{j})} \triangleq g\ddx + \sum_{i=1}^{m}\sum_{p=0}^{j_{i}}u_{i}^{(p+1)}\ddu{i}{p}, \qquad
\gamma^{(\mbf{l})} \triangleq \gamma\ddy + \sum_{i=1}^{m}\sum_{p=0}^{l_{i}}v_{i}^{(p+1)}\ddv{i}{p}.
$$
The prolonged states are, respectively, $\ol{x}^{(\mbf{j})}$ and $\ol{y}^{(\mbf{l})}$, and the control inputs $u^{(\mbf{j+1})}$ and $v^{(\mbf{l+1})}$.
\begin{defn}\label{PPequiv:def}
The systems \eqref{sys-equiv:def} are equivalent by pure prolongation (in short $P^2$-equivalent) at a point $\ol{x}_{0} \in X\times \R^{m}_{\infty}$ if, and only if, there exist finite $\mbf{j}$ and $\mbf{l}$ such that the prolonged systems of order $\mbf{j}$ and $\mbf{l}$ respectively are equivalent by diffeomorphism and feedback, \ie if, and only if, there exists a local diffeomorphism $\varphi$  
 and a feedback $W$:
$$
\ol{y}^{(\mbf{l})} = \varphi( \ol{x}^{(\mbf{j})}) \qquad
v^{(\mbf{l+1})} = W(\ol{x}^{(\mbf{j})}, u^{(\mbf{j+1})})
$$
with $W$ invertible with respect to $u^{(\mbf{j+1})}$ for all $\ol{x}^{(\mbf{j})}$ in a suitable neighborhood of $\ol{x}_{0}$, such that 
$$\gamma^{(\mbf{l})} \circ (\varphi\times W) = \varphi_{\star}g^{(\mbf{j})}.$$
%
\end{defn}
This equivalence indeed implies that $n + \mid \mbf{j} \mid = n' + \mid \mbf{l} \mid$.

\begin{rem}
The equivalence relation by \emph{diffeomorphism and feedback} is easily seen to be strictly finer than the \emph{$P^2$-equivalence} (take $\mbf{j}=\mbf{l}=\0$), which in turn is strictly finer than the \emph{Lie-B\"acklund equivalence} (see \eg Example~\ref{pend:ex}).
\end{rem}

\begin{defn}\label{flatbypp:def}
We say that system~\eqref{control-aff-sys:eq} is \emph{flat by pure prolongation} (in short $P^2$-flat) at a point $\ol{x}_{0} \in X\times \R^{m}_{\infty}$ if, and only if, it is $P^2$-equivalent to $\dot{y}_{i} = v_{i}$, $i=1, \ldots,m$ and $y$ is called $P^2$-flat output.
\end{defn}

It is therefore immediate to remark that a system is $P^2$-flat if, and only if,
there exists a pure prolongation of finite order $\mbf{j}$ such that the prolonged system is feedback linearizable at $\ol{x}_{0}$, thus recovering the definition of linearization by prolongation already introduced in  \cite{JF,FrFo-ejc2005,BC-scl2004}. 

Moreover, a $P^2$-flat output being obviously a flat output and pure prolongations being particular cases of Lie-B\"acklund isomorphisms, the class of $P^2$-flat systems is indeed contained in the class of Lie-B\"acklund equivalence to $0$ (modulo the trivial field $\tau$), \ie  constitutes a subclass of differentially flat systems.

The next Lemma extends a well-know result (see \eg \cite{CLMscl,Sluis-scl93}) to our context (see also \cite{FrFo-ejc2005}).

\begin{lem}\label{reduc-j:lem}
We consider system \eqref{nlsys(j):def}, denoted by $\Sigma^{(\mbf{j})}$, with $\mbf{j} = (j_{1}, \ldots, j_{m})$, assuming, without loss of generality, that $0 \leq j_{1} \leq \ldots \leq j_{m}$, possibly up to input renumbering. We denote by $\mbf{j'} = \mbf{j-j_{1}} = (0, j_{2}-j_{1} \ldots, j_{m}-j_{1})$, and by $\Sigma^{(\mbf{j'})}$ the corresponding system. Then $\Sigma^{(\mbf{j})}$ is $P^2$-flat at a given point $(x_{0},\ol{u}_{0})$ if, and only if, $\Sigma^{(\mbf{j'})}$ is also $P^2$-flat at this point. Moreover, every $P^2$-flat output of $\Sigma^{(\mbf{j})}$ at $(x_{0},\ol{u}_{0})$ is a $P^2$-flat output of $\Sigma^{(\mbf{j'})}$ at the same point, and conversely.

\end{lem}
\begin{proof}
Since $\Sigma^{(\mbf{j})}$ is $P^2$-flat, there exists $\mbf{j_{0}}$ and $\mbf{l_{0}}$ such that $\Sigma^{(\mbf{j + j_{0}})}$ is feedback equivalent to the linear system $y^{(\mbf{l_{0}+1})}= w$. Therefore, since, by assumption, $\Sigma^{(\mbf{j + j_{0}})} = \Sigma^{(\mbf{j' + j_{1} +j_{0}})}$, $\Sigma^{(\mbf{j'})}$ is also $P^2$-flat.
The converse is trivial as well as the fact that $\Sigma^{(\mbf{j + j_{0}})}$ and $\Sigma^{(\mbf{j' + j_{1} +j_{0}})}$ have the same $P^2$-flat output.
\end{proof}

\begin{rem}
In the single input case ($m=1$), this lemma shows that a $P^2$-flat system is equivalent, without prolongation, to a linear system. We thus trivially recover the fact that $P^2$-flatness and feedback linearizability are equivalent in this case (see \cite{CLMscl} and remark~\ref{singleinputequiv:rem})
\end{rem}

\subsection{Necessary and Sufficient Conditions}

The following Proposition is a straightforward adaptation of Theorem~\ref{feedlin-thm} for an arbitrary order $\bj$. Note that, at this stage, nothing is said about a possible choice of $\bj$, a question that will be dealt with in subsection~\ref{prolongalg:subsec}, theorem~\ref{alg-proof:thm}.

\begin{prop}\label{lineariz-j:prop}
The prolonged system of order $\mbf{j}$ is feedback linearizable at $\ol{0}$ if, and only if, 
$G_{k}^{(\mbf{j})}$ is involutive with locally constant dimension for all $k$ and such that  $G_{k^{(\mbf{j})}_{\ast}}^{(\mbf{j})} = \T X^{(\bj)}$. 
\end{prop}

Again, flat outputs can be computed via Frobenius theorem, once established the list of \emph{Brunovsk\'{y}'s controllability indices of order $\bj$}, as follows:
\begin{defn}\label{brunov-j:def}
Consider the sequence of integers 
$$\rho_{k}^{(\bj)} \triangleq \dim  G_{k}^{(\bj)} / G_{k-1}^{(\bj)}  \quad \forall k\geq 1,  \qquad \rho_{0}^{(\bj)} \triangleq \dim G_{0}^{(\bj)} = m.$$
The Brunovsk\'{y} controllability indices of order $\bj$ are defined by
$$\kappa_{k}^{(\bj)} \triangleq \# \{ l \mid \rho_{l}^{(\bj)} \geq k \}, \quad k= 1,\ldots, m,$$
\end{defn}

As in the $0$th order case, if the prolonged system of order $\mbf{j}$ is feedback linearizable at $\ol{0}$, we have:
\begin{itemize}
\item $\rho_{k}^{(\bj)}$'s and $\kappa_{k}^{(\bj)}$'s are non increasing sequences, 
\item $\rho_{k}^{(\bj)}\leq m$  for all $k$ and $\rho_{k}^{(\bj)}=0$ for all $k\geq k^{(\bj)}_{\ast}+1$,
\item $\kappa_{1}^{(\bj)} = k^{(\bj)}_{\ast} + 1, \quad \kappa_{m}^{(\bj)} \geq 1$, 
\item $\ds \sum_{k=0}^{k^{(\bj)}_{\ast}} \rho_{k}^{(\bj)} = \sum_{k=1}^{m} \kappa_{k}^{(\bj)} =  \dim G_{k^{(\bj)}_{\ast}}^{(\bj)}  = n+m + \mid \bj \mid$.
\end{itemize}

The list $\kappa_{1}^{(\bj)}, \ldots, \kappa_{m}^{(\bj)}$ is uniquely defined up to input permutation, invariant by  prolonged state feedback and prolonged state diffeomorphism,  and is thus equal to the list of controllability indices of the associated linear system \eqref{lin-sys:eq} with $r_{i}= \kappa_{i}^{(\bj)}$, $i=1, \ldots, m$. 

Moreover, for all $k$ and all $i=1,\ldots, m$, and possibly up to a suitable input reordering, we have 
$$G_{k}^{(\bj)} = \bigoplus_{p=1}^{m}  \left\{\ad_{g_{0}^{(\bj)}}^{l}\ddu{p}{j_{p}} \mid l= 0, \ldots k\lor (\kappa_{p}^{(\bj)}-1)\right\}, \quad G_{\kappa_{1}^{(\bj)}-1}^{(\bj)} = G_{k^{(\bj)}_{\ast}}^{(\bj)} = \T X^{(\bj)}$$
and flat outputs $(y_{1}, \ldots, y_{m})$ are locally non trivial solutions of the system of PDE's
\begin{equation}\label{PDEsj:eq}
\left< G_{k}^{(\bj)},  dy_{i} \right>= 0, \; k= 0,\ldots, \kappa_{i}^{(\bj)}-2, \quad \mathrm{with~}                           
\quad \left< G_{\kappa_{i}^{(\bj)}-1}^{(\bj)},  dy_{i} \right>\neq 0, \quad i=1, \ldots,m.
\end{equation}
Finally, the mapping
$$\ol{x}^{(\bj)} \mapsto (y_{1}, \ldots, y_{1}^{(\kappa_{1}^{(\bj)}-1)}, \ldots , y_{m}, \ldots, y_{m}^{(\kappa_{m}^{(\bj)}-1)})$$
is a local diffeomorphism.

\bigskip

In virtue of Lemma~\ref{reduc-j:lem}, it suffices to restrict our analysis to prolongations  of order $\mbf{j}= (j_{1}, \ldots, j_{m})$ such that $0=j_{1}\leq \ldots, \leq j_{m}$. 

\bigskip
We are now ready to state our main result.
\begin{thm}\label{cns0:thm}
A necessary and sufficient condition for $P^2$-flatness at $\ol{0}$ is that there exists $\mbf{j}= (j_{1}, \ldots, j_{m}) \in \N^{m}$, $0=j_{1}\leq \ldots, \leq j_{m}$, such that
\begin{itemize}
\item[(i)] $\ol{\Delta_{k}^{(\mbf{j})}} = \Delta_{k}^{(\mbf{j})}$ with $\dim \Delta_{k}^{(\mbf{j})}$  locally constant for all $k\geq 0$,
\item[(ii)] $[\Gamma_{k}^{(\mbf{j})}, \Delta_{k}^{(\mbf{j})}] \subset  \Delta_{k}^{(\mbf{j})}$ for all $k\geq 0$, 
\item[(iii)] $k^{(\bj)}_{\ast}$ is such that $\Delta_{k}^{(\mbf{j})} = \T X\times \T\R^{m}$ and $\Gamma_{k}^{(\mbf{j})} = \T\R^{\mid \mbf{j} \mid}$ for all $k\geq  k^{(\bj)}_{\ast}$.
\end{itemize}
\end{thm}

\begin{proof}
By \eqref{G-k-j:eq} of Proposition~\ref{G-k-j:prop}, we have 
$G_{k}^{(\mbf{j})} =  \Gamma_{k}^{(\mbf{j})} \oplus \Delta_{k}^{(\mbf{j})}$  for all $k\geq 0$.
Then, $G_{k}^{(\mbf{j})} = \ol{G_{k}^{(\mbf{j})}}$ implies that 
$[\Gamma_{k}^{(\mbf{j})} \oplus \Delta_{k}^{(\mbf{j})}, \Gamma_{k}^{(\mbf{j})} \oplus \Delta_{k}^{(\mbf{j})}] \subset \Gamma_{k}^{(\mbf{j})} \oplus \Delta_{k}^{(\mbf{j})}$. Since $\Gamma_{k}^{(\mbf{j})} = \ol{\Gamma_{k}^{(\mbf{j})}}$ for all $k$, and since $[\Gamma_{k}^{(\mbf{j})} , \Delta_{k}^{(\mbf{j})}] \cap \Gamma_{k}^{(\mbf{j})}  = \{0\}$ by Lemma~\ref{prepa2:lem} and \eqref{Gammaprop:eq}, we deduce that $[\Gamma_{k}^{(\mbf{j})}, \Delta_{k}^{(\mbf{j})}]  + [\Delta_{k}^{(\mbf{j})} , \Delta_{k}^{(\mbf{j})}]  \subset \Delta_{k}^{(\mbf{j})}$, hence $[\Gamma_{k}^{(\mbf{j})}, \Delta_{k}^{(\mbf{j})}] \subset  \Delta_{k}^{(\mbf{j})}$
and $[\Delta_{k}^{(\mbf{j})} , \Delta_{k}^{(\mbf{j})}]  \subset \Delta_{k}^{(\mbf{j})}$, \ie $\Delta_{k}^{(\mbf{j})} = \ol{\Delta_{k}^{(\mbf{j})}}$, for all $k\geq 0$. 

Conversely, $[\Gamma_{k}^{(\mbf{j})}, \Delta_{k}^{(\mbf{j})}] \subset  \Delta_{k}^{(\mbf{j})}$
and $\Delta_{k}^{(\mbf{j})} = \ol{\Delta_{k}^{(\mbf{j})}}$ for all $k\geq 0$ trivially implies that $G_{k}^{(\mbf{j})} = \ol{G_{k}^{(\mbf{j})}}$  for all $k\geq 0$.

Moreover, since $\dim \Gamma_{k}^{(\mbf{j})}$ is constant by construction, the fact that $G_{k}^{(\mbf{j})}$ has locally constant dimension is equivalent to the fact that $\Delta_{k}^{(\mbf{j})}$ has locally constant dimension too for all $k\geq 0$, hence (i).

Finally, (iii) is an immediate consequence of the condition that $G_{k}^{(\mbf{j})} = \T X\times \T\R^{m+\mid \mbf{j} \mid}$ for all $k\geq  k^{(\bj)}_{\ast}$, and the theorem is proven.
\end{proof}
\begin{rem}
Expressed in words, the above necessary and sufficient conditions read (i) involutivity with locally constant dimension of the $\Delta_{k}^{(\mbf{j})}$'s, (ii) invariance of $\Delta_{k}^{(\mbf{j})}$ by $\Gamma_{k}^{(\mbf{j})}$ and (iii) strong controllability rank condition (see \eg \cite{sussmann-jurdjevic,hermann-krener}).
\end{rem}

\subsection{The Pure Prolongation Algorithm}\label{prolongalg:subsec}
From now on, for every sequence $\mbf{l} = \left(l_{1}, \ldots, l_{m}\right)\in \N^{m}$, we systematically re-order the indices $\{1,\ldots, m\}$ by a suitable permutation $\alpha : \{1,\ldots, m\} \rightarrow \{1,\ldots, m\}$ such that 
$0= l_{\alpha(1)} \leq  \ldots \leq l_{\alpha(m)}$.
Moreover, for simplicity's sake, the permutation $\alpha$ will be omitted. We will thus abusively replace $l_{\alpha(i)}$ by $l_{i}$, \ie $0=l_{1} \leq  \ldots \leq l_{m}$.


In the following algorithm we assume that the computations are done in a suitable open dense neighborhood of $\ol{0}$ where all the distributions involved have constant dimension.


\subsubsection{Initialization}

Consider the filtration $\{ G_{k}^{(0)} \mid k\geq 0\}$ defined by \eqref{dist0:def}. If every $G_{k}^{(0)}$ satisfies the conditions of theorem~\ref{feedlin-thm}, the system is feedback linearizable and no prolongation is needed. In particular, for $m=1$, a case where flatness and feedback linearizability are equivalent, the results of this section are pointless.

Otherwise, we have the following alternative:
\begin{itemize}
\item either there must exist $k_0  \leq n$ such that $G_{k_{0}}^{(0)}$ is not involutive while every $G_{k}^{(0)}$ is involutive for $k<k_{0}$ (note that $k_0 \geq 1$ since $G_{0}^{(\0)}$ is always involutive), 
\item or the $G_{k}^{(0)}$'s are all involutive but with $\max_{k\geq 0}\dim G_{k}^{(0)} = \dim G_{k_{\star}^{(\0)}}^{(0)} < n+m$.
\end{itemize}

\paragraph{First Case.}
There exists $k_0  \in\{1, \ldots, n\}$, first index for which $G_{k_{0}}^{(0)}$ is not involutive. 

Let $H_{1}^{(\0)}\neq \{0\}$ be any involutive distribution included in $G_{1}^{(\0)}$, of the form
\begin{equation}\label{H1-0:eq}
H_{1}^{(\0)} \triangleq \{\ddu{1}{0}, \ldots, \ad_{g_{0}^{(0)}} \ddu{1}{0}, \ldots, \ddu{p_{0}}{0}, \ldots, \ad_{g_{0}^{(0)}} \ddu{p_{0}}{0} \}
\end{equation} 
for some $p_{0}\leq m$. 

By \eqref{adkg0j-adq0gij+:eq:appndx}, $\ad_{g_{0}^{(0)}} \ddu{p}{0} = \ad_{g_{0}^{(\mbf{l})}} \ddu{p}{0}$ for all $\mbf{l}$ and, thanks to \eqref{0brack:eq}, we have $[ \ddu{q}{l_{q}-k}, \ad_{g_{0}^{(0)}} \ddu{p}{0}] = 0$ for all $l_{q}-k > 1$ with $q > p_{0}$ and all $p=1, \ldots, p_{0}$. Therefore we can choose
\begin{equation}\label{Delta1-0:eq}
\Delta_{1}^{(\mbf{l})} \triangleq \{\ddu{1}{0}, \ldots, \ad_{g_{0}^{(\mbf{l})}} \ddu{1}{0}, \ldots, \ddu{p_{0}}{0}, \ldots, \ad_{g_{0}^{(\mbf{l})}} \ddu{p_{0}}{0} \} = H_{1}^{(\0)}
\end{equation}
and 
\begin{equation}\label{Gamma1-0:eq}
\Gamma_{1}^{(\mbf{l})} \triangleq \{\ddu{p_{0}+1}{l_{p_{0}+1}},\ddu{p_{0}+1}{l_{p_{0}+1}-1} , \ldots, \ddu{m}{l_{m}},\ddu{m}{l_{m}-1} \}
\end{equation}
for any 
\begin{equation}\label{ll-init:eq}
\mbf{l} = \left(\underbrace{0,\ldots, 0}_{p_{0}}, l_{p_{0}+1}, \ldots, l_{m}\right), \quad 2 = l_{p_{0}+1} \leq \ldots \leq l_{m}
\end{equation}
since they naturally satisfy $\ol{\Delta_{1}^{(\mbf{l})}}= \ol{H_{1}^{(\0)}} = H_{1}^{(\0)} = \Delta_{1}^{(\mbf{l})}$ and 
$[\Gamma_{1}^{(\mbf{l})},\Delta_{1}^{(\mbf{l})}] \subset \Delta_{1}^{(\mbf{l})}$ for all $\mbf{l}$ given by \eqref{ll-init:eq}, \ie with $l_{p_{0}+1} = 2$. 

However if, for $l_{p_{0}+1} = 1$,
\begin{equation}\label{Delta1-0+:eq}
\Delta_{1}^{(\mbf{l})} \triangleq \{\ddu{1}{0}, \ldots, \ad_{g_{0}^{(\mbf{l})}} \ddu{1}{0}, \ldots, \ddu{p_{0}}{0}, \ldots, \ad_{g_{0}^{(\mbf{l})}} \ddu{p_{0}}{0}, \ddu{p_{0}+1}{0} \}
\end{equation}
and
\begin{equation}\label{Gamma1-0+:eq}
\Gamma_{1}^{(\mbf{l})} \triangleq \{\ddu{p_{0}+1}{1},\ddu{p_{0}+2}{l_{2}}, \ddu{p_{0}+2}{l_{2}-1}, \ldots, \ddu{m}{l_{m}},\ddu{m}{l_{m}-1} \}
\end{equation}
also satisfy  $\ol{\Delta_{1}^{(\mbf{l})}} = \Delta_{1}^{(\mbf{l})}$ and $[\Gamma_{1}^{(\mbf{l})},\Delta_{1}^{(\mbf{l})}] \subset \Delta_{1}^{(\mbf{l})}$, the previous choice \eqref{Delta1-0:eq}-\eqref{Gamma1-0:eq}-\eqref{ll-init:eq} maybe replaced by \eqref{Delta1-0+:eq}-\eqref{Gamma1-0+:eq} with $1=l_{p_{0}+1}\leq \ldots \leq l_{m}$.

These choices, namely  \eqref{Delta1-0:eq}-\eqref{Gamma1-0:eq}-\eqref{ll-init:eq}, or \eqref{Delta1-0+:eq}-\eqref{Gamma1-0+:eq} with $ l_{p_{0}+1} =1$, for any $H_{1}^{(\0)}$ given by \eqref{H1-0:eq}, constitutes the initialization of the algorithm. 

Note that there are at most $\sum_{p=1}^{m-1} {\mathrm{C}}_{m}^{p} = \sum_{p=1}^{m-1} \frac{m!}{p!(m-p)!}$ possibilities of such initialization.

\paragraph{Second Case.}
\begin{lem}\label{strcontG:lem}
Assume that the $G_{k}^{(0)}$'s are all involutive with $\max_{k\geq 0}\dim G_{k}^{(\0)} < n+m$. Then
$G_{k}^{(\mbf{j})} = \ol{G_{k}^{(\mbf{j})}} = G_{k}^{(\0)}$ for all $k$ and $\mbf{j}$ and $\max_{k\geq 0} \dim G_{k}^{(\mbf{j})} < m+n$ for all $\mbf{j}$.
\end{lem}
\begin{proof}
According to \eqref{adkg0j-adq0gij+:eq:appndx} of lemma~\ref{prepa1:lem},  every $\ad_{g_{0}^{(\mbf{j})}}^{k}\ddu{p}{j_{p}}$ is equal to $\pm \ad_{g_{0}^{(\0)}}^{k-j_{p}}\ddu{p}{0}$ + terms in $\ol{G_{k-1}^{(\0)}} = G_{k-1}^{(\0)}$, thanks to the involutivity of the $G_{k}^{(0)}$'s. Therefore, $G_{k}^{(\mbf{j})} \subset G_{k}^{(\0)}$, the converse inclusion being immediate using \eqref{adkg0j-adq0gij+:eq:appndx}. Thus $\dim G_{k}^{(\0)} = \dim G_{k}^{(\mbf{j})}$ for all $k$ and $\mbf{j}$, Q.E.D.
\end{proof}
We immediately conclude that, in this case, the strong controllability rank condition does not hold, which contradicts (iii) of theorem~\ref{feedlin-thm}, hence the non-flatness by pure prolongation. Therefore, this case must be discarded.

\subsubsection{Recursion}

For all $k\geq 1$ and all $\mbf{l}$ given by:
\begin{equation}\label{l-init:eq}
\mbf{l} = \left(\underbrace{0,\ldots, 0}_{p_{0}}, l_{p_{0}+1}, \ldots, l_{m}\right), \quad 1 \leq l_{p_{0}+1} \leq \ldots \leq l_{m},
\end{equation}
and denoting by $\Cmin_{\mbf{l}\in L} (l_{p_{0}+1}, \ldots, l_{m})$ (resp. $\Cmax_{\mbf{l}\in L} (l_{p_{0}+1}, \ldots, l_{m})$) the componentwise minimum (resp. maximum) of a collection of $(m-p_{0})$-tuples $(l_{p_{0}+1}, \ldots, l_{m})$ in a set $L$, with $1 \leq l_{p_{0}+1} \leq \ldots \leq l_{m}$, let us introduce the following numbers 
\begin{equation}\label{sig-Del:def}
\begin{aligned}
\sigma_{\Delta}(k) &\triangleq \left(\sigma_{p_{0}+1,\Delta}(k), \ldots, \sigma_{m,\Delta}(k)\right) \\
&\triangleq   \Cmin\{ 1 \leq l_{p_{0}+1}, \leq \ldots \leq l_{m} \mid  \ol{\Delta_{k}^{(\mbf{l})}}= \Delta_{k}^{(\mbf{l})}\},
\end{aligned}
\end{equation}
\begin{equation}\label{sig-Gam-Del:def}
\begin{aligned}
\sigma_{\Gamma,\Delta}(k) &\triangleq \left( \sigma_{p_{0}+1,\Gamma,\Delta}(k), \ldots, \sigma_{m,\Gamma,\Delta}(k) \right)\\
&\triangleq \Cmin\{ 1 \leq  l_{p_{0}+1} \leq \ldots \leq l_{m} \mid [\Gamma_{k}^{(\mbf{l})}, \Delta_{k}^{(\mbf{l})}] \subset \Delta_{k}^{(\mbf{l})}\}
\end{aligned}
\end{equation}
and
\begin{equation}\label{sig-max:def}
\begin{aligned}
&\left( j_{p_{0}+1}, \ldots, j_{m}\right) \\
&\qquad \triangleq \Cmax_{k\geq 1} \left( \sigma_{p_{0}+1,\Delta}(k) \wedge \sigma_{p_{0}+1,\Gamma,\Delta}(k) \leq \ldots, \leq \sigma_{m,\Delta}(k) \wedge \sigma_{m,\Gamma,\Delta}(k)\right).
\end{aligned}
\end{equation}
\begin{rem}
Thanks to \eqref{ll-init:eq}-\eqref{Delta1-0:eq}-\eqref{Gamma1-0:eq}, we get
$$\sigma_{p_{0}+1,\Delta}(1) \wedge \sigma_{p_{0}+1,\Gamma,\Delta}(1)\leq 2$$
and if $\ol{\Delta_{1}^{(\mbf{l})}} = \Delta_{1}^{(\mbf{l})}$ and $[\Gamma_{1}^{(\mbf{l})},\Delta_{1}^{(\mbf{l})}] \subset \Delta_{1}^{(\mbf{l})}$, with \eqref{Delta1-0+:eq}-\eqref{Gamma1-0+:eq} and $l_{p_{0}+1}, \ldots, l_{m}\geq 1$, we get 
$$\sigma_{p_{0}+1,\Delta}(1) \wedge \sigma_{p_{0}+1,\Gamma,\Delta}(1) = 1.$$
\end{rem}

\begin{rem}\label{indep:rem}
In view of the definition \eqref{Deltadef:eq} of $\Delta_{k}^{(\mbf{l})}$, if $l_{q} \leq k < l_{q+1}$ for some $q\in \{1, \ldots,m\}$ and $k\geq 1$, $\Delta_{k}^{(\mbf{l})}$ does not depend on $l_{q+1}, \ldots, l_{m}$. Thus $\sigma_{\Delta}(k) = \left( \sigma_{1,\Delta}(k), \ldots, \sigma_{q,\Delta}(k), \underbrace{\sigma_{q,\Delta}(k), \ldots, \sigma_{q,\Delta}(k)}_{m-q}\right)$.
\end{rem}

\begin{prop}\label{sigma-p0+1-bound:prop}
We have
\begin{equation}\label{sigma-p0+1:ineq}
[\Gamma_{k}^{(\mbf{l})}, \Delta_{k}^{(\mbf{l})}]= \{0\}, \quad \mathrm{and} \quad   \sigma_{p_{0}+1,\Gamma,\Delta}(k)\leq 2k \quad \forall l_{p_{0}+1} > 2k-1, \; \forall k\geq 0.
\end{equation}
Moreover, $\Delta_{k}^{(\mbf{l})}$ is independent of $l_{p_{0}+1}$ if $l_{p_{0}+1} > k$. Therefore, if $\Delta_{k}^{(\mbf{l})}$ is involutive for all $k$, we get
\begin{equation}\label{sigma-k-p0+1:ineq}
\sigma_{p_{0}+1,\Delta}(k) \wedge \sigma_{p_{0}+1,\Gamma,\Delta}(k)\leq 2k
\end{equation}
and $j_{p_{0}+1} \leq \ldots \leq j_{m} < + \infty$. The minimal prolongation orders are thus obtained by minimizing \eqref{sig-max:def} over all initializations built on \eqref{H1-0:eq}.
\end{prop}
\begin{proof} \eqref{sigma-k-p0+1:ineq} is an immediate consequence of \eqref{0brack:eq} and of the definitions of  $\Delta_{k}^{(\mbf{l})}$ and $\Gamma_{k}^{(\mbf{l})}$.
Moreover, since, according to remark~\ref{nondecrease:rem}, $k\mapsto \dim\Delta_{k}^{(\mbf{l})}$ is non decreasing and bounded by $n+m$, then there exists a finite $K^{\star}$ such that $\Delta_{K^{\star}}^{(\mbf{l})} = \Delta_{k}^{(\mbf{l})}$ for all $k\geq K^{\star}$ and thus the maximum with respect to $k$ in 
\eqref{sig-max:def} is achieved for $k = K^{\star}$. The C-minimal prolongation orders are thus obtained by varying the initializations built on \eqref{H1-0:eq}.
\end{proof}
We thus have proven:
\begin{thm}\label{alg-proof:thm} We have the following alternative:
\begin{enumerate}
\item If, for every choice of initialization built on \eqref{H1-0:eq}, there exists a $k$ for which $\max_{p\in \{p_{0}+1, \ldots,m\}}\sigma_{p,\Delta}(k) = +\infty$ or if  $\max_{k} \dim \Delta_{k}^{(\mbf{j})} < n+m$ for all $\mbf{j}$, then the system is not flat by pure prolongation.
\item Otherwise, $\mbf{j}$ is finite and given by \eqref{sig-max:def} and the minimal prolongation orders are  obtained by C-minimizing \eqref{sig-max:def} over all initializations.
\end{enumerate}
\end{thm}
\begin{proof}Straightforward from what precedes.
\end{proof}
We immediately deduce the following algorithm:
\begin{algo}[flatness by pure prolongation]\label{algo1}
\begin{description}
 \item[Input:]the vector fields $g_{0}^{(\0)}, g_{1}^{(0)} = \ddu{1}{0}, \ldots, g_{m}^{(0)} = \ddu{m}{0}$ (see \eqref{prolong1vecf:eq})
\item[output:] the minimal lengths $0=j_{1}\leq \ldots \leq j_{m}$ and $k_{\star}^{(\mbf{j})}$ or \textsc{fail} if the system is not flat by pure prolongation.
\item[Initialization.] Construct $G_{k}^{(0)}$ for all $k \leq n$.
If $\exists k_0  \leq n$ such that $G_{k_{0}}^{(0)}$ is not involutive, choose an involutive subdistribution $H_{1}^{(\0)}\subset G_{k_{0}}^{(0)}$, given by \eqref{H1-0:eq}, with $p_{0}\leq m-1$ and compute $\Delta_{1}^{(\mbf{l})}$ and $\Gamma_{1}^{(\mbf{l})}$ by \eqref{Delta1-0:eq}-\eqref{Gamma1-0:eq}-\eqref{ll-init:eq}, or \eqref{Delta1-0+:eq}-\eqref{Gamma1-0+:eq} with $ l_{p_{0}+1} =1$.
\item[Step $k\geq 1 $.] Compute  $\Delta_{k}^{(\mbf{l})}$, $\Gamma_{k}^{(\mbf{l})}$, $\sigma_{\Delta}(k)$ and $\sigma_{\Gamma,\Delta}(k)$ by \eqref{sig-Del:def} and \eqref{sig-Gam-Del:def}. Continue up to the first $k_{1}$ such that $\sigma_{\Gamma,\Delta}(k_{1})\bigwedge \sigma_{\Delta}(k_{1})$ is C-maximal. Then $\mbf{j}$ is given by \eqref{sig-max:def}.
\\
If $\Delta_{k}^{(\mbf{l})}$ is non involutive for some $k \leq k_{1}$, change the initialization and restart and if $\Delta_{k}^{(\mbf{l})}$ is non involutive for all initializations, then \textsc{fail}.
\item[Step $k_{\star}^{(\mbf{j})}$.] Determine $k_{\star}^{(\mbf{j})}$. If $\dim \Delta_{k_{\star}^{(\mbf{j})}}^{(\mbf{j})} =n+m$, \textsc{stop}. Otherwise, \textsc{fail}.
 \end{description}

\end{algo}

\section{Examples}\label{5ex:ex}

All the examples of this paper concern two input systems, \ie $m=2$, except example \ref{3in:ex}, with $3$ inputs, resuming the example 2 of \cite{KLO_ijrnc20}. 

In the two input examples, the prolongation index $\mbf{j}= (j_{1}, j_{2})$ is supposed to be such that $0= j_{1}\leq j_{2}$, up to a suitable input permutation. However, for the sake of readability, we will keep the original input numbering unchanged and thus consider that $\mbf{j}= (0,j_{2})$ or $(j_{1}, 0)$ depending on the context. At the exception of this modification, we strictly apply algorithm~\ref{algo1} in all the examples. 

The first example gives a detailed presentation of the application of this algorithm, in particular the role played by $\sigma_{\Gamma,\Delta}(k)$ and $\sigma_{\Delta}(k)$. The second one shows the importance of the sole number $\sigma_{1,\Delta}(k)$ to determine the prolongation length, and the third one, borrowed from  \cite{CLM}, and carried on again in \cite[Section 3.1]{BC-scl2004}, is reported here to compare our approach with the one of \cite{BC-scl2004}. Then, the pendulum example is presented to show that non flat systems by pure prolongation can be detected in a finite number of steps.
To conclude this section, the last example, with 3 inputs, compares the constuction proposed in \cite[Example 2]{KLO_ijrnc20} with our algorithm leading to a minimal prolongation, strictly smaller than the one of \cite{KLO_ijrnc20}.

\subsection{Chained System \cite{MMR-ecc}}\label{chainex:subsec}

\begin{equation}\label{chainsys:eq}
\begin{aligned}
&x_{1}^{(3)} = u_{1}\\
&\ddot{x}_{2} = u_{2}\\
&\dot{x}_{3} = u_{1}u_{2}
\end{aligned}
\end{equation}

This system has been proven to be flat in \cite[section 3.1.1]{MMR-ecc} with the following flat output 
\begin{equation}\label{chain-f-out:eq}
y_{1}= x_{3} - \ddot{x}_{1}u_{2} + \dot{x}_{1}\dot{u}_{2} - x_{1}\ddot{u}_{2}, \quad y_{2} = x_{2}.
\end{equation}

\subsubsection{Associated Non Prolonged Distributions}
Let us start this section by showing that system~\eqref{chainsys:eq} is not static feedback linearizable.
We denote the state coordinates by $(x_{1,1}, x_{1,2}, x_{1,3}, x_{2,1}, x_{2,2}, x_{3}, u_{1}^{(0)}, u_{2}^{(0)})$ ($n=6$ and $m=2$), with $x_{1,j} \triangleq x_{1}^{(j)}$, $j=1,2,3$, $x_{2,j} \triangleq x_{2}^{(j)}$, $j=1,2$ and $u_{i}^{(0)} = u_{i}$, $i=1,2$. 
\\
The two input variables are $u_{i}^{(1)} \triangleq \dot{u}_{i}$, $i=1,2$.

The system vector fields are
\begin{equation}\label{vecf-ex-chainsys:eq}
\begin{aligned}
g_{0}^{(0)} &\triangleq x_{1,2}\frac{\partial}{\partial x_{1,1}} + x_{1,3}\frac{\partial}{\partial x_{1,2}} +
u_{1}^{(0)}\frac{\partial}{\partial x_{1,3}} 
+ x_{2,2}\frac{\partial}{\partial x_{2,1}} + u_{2}^{(0)}\frac{\partial}{\partial x_{2,2}} + u_{1}^{(0)}u_{2}^{(0)}\frac{\partial}{\partial x_{3}}\\
g_{1}^{(0)} &\triangleq \ddu{1}{0},\qquad 
g_{2}^{(0)} \triangleq \ddu{2}{0}
\end{aligned}
\end{equation}
One can verify that 
$$\begin{aligned}
&\ad_{g_{0}^{(0)}} g_{1}^{(0)} = - \frac{\partial}{\partial x_{1,3}} - u_{2}^{(0)}\frac{\partial}{\partial x_{3}}, \quad
\ad_{g_{0}^{(0)}}^{2} g_{1}^{(0)} = \frac{\partial}{\partial x_{1,2}},  \quad
\ad_{g_{0}^{(0)}}^{3} g_{1}^{(0)} = - \frac{\partial}{\partial x_{1,1}}, \quad 
\ad_{g_{0}^{(0)}}^{4} g_{1}^{(0)} = 0
\end{aligned}$$
and
$$\begin{aligned}
&\ad_{g_{0}^{(0)}} g_{2}^{(0)} = - \frac{\partial}{\partial x_{2,2}} - u_{1}^{(0)}\frac{\partial}{\partial x_{3}}, \quad \ad_{g_{0}^{(0)}}^{2} g_{2}^{(0)}  = \frac{\partial}{\partial x_{2,1}}, \quad \ad_{g_{0}^{(0)}}^{3} g_{2}^{(0)}  = 0.
\end{aligned}$$

Therefore
$$G_{0}^{(0)}= \left\{\frac{\partial}{\partial u_{1}^{(0)}}, \frac{\partial}{\partial u_{2}^{(0)}}\right\} = \ol{G_{0}^{(0)}},$$
$$G_{1}^{(0)}=  G_{0}^{(0)} + \left\{ - \frac{\partial}{\partial x_{1,3}} - u_{2}^{(0)}\frac{\partial}{\partial x_{3}}, - \frac{\partial}{\partial x_{2,2}} - u_{1}^{(0)}\frac{\partial}{\partial x_{3}}\right\} \neq \ol{G_{1}^{(0)}},$$
since, \eg, $[g_{2}^{(0)}, \ad_{g_{0}^{(0)}} g_{1}^{(0)} ]=-\frac{\partial}{\partial x_{3}} \not\in G_{1}^{(0)}$, and $\dim G_{1}^{(0)} = 4$ whereas $\dim \ol{G_{1}^{(0)}} = 5$,
$$G_{2}^{(0)}= G_{1}^{(0)} + \left\{\frac{\partial}{\partial x_{1,2}}, \frac{\partial}{\partial x_{2,1}}\right\} \neq \ol{G_{2}^{(0)}}, \quad \dim\ol{G_{2}^{(0)}} = 7,
$$
$$
G_{3}^{(0)}= G_{2}^{(0)} + \left\{\frac{\partial}{\partial x_{1,1}}\right\} \neq \ol{G_{3}^{(0)}}
$$
and $G_{k}^{(0)}= G_{3}^{(0)}$ for all $k \geq 3$. Moreover, $\dim{G_{3}^{(0)}}= 7 < \dim{ \ol{G_{3}^{(0)}}} = n+m = 8$. We conclude that the system is not feedback linearizable.

Since $G_{1}^{(0)}$ is not involutive, there are two possible initializations: 
\begin{equation}\label{init1ex1:eq}
H_{1,1}^{(\0)}= 
\{ g_{1}^{(0)}, \ad_{g_{0}^{(0)}} g_{1}^{(0)} \} = \{ \ddu{1}{0},  - \frac{\partial}{\partial x_{1,3}} - u_{2}^{(0)}\frac{\partial}{\partial x_{3}} \} = \ol{H_{1,1}^{(\0)}}
\end{equation}
or
\begin{equation}\label{init2ex1:eq}
H_{1,2}^{(\0)}= 
\{ g_{2}^{(0)}, \ad_{g_{0}^{(0)}} g_{2}^{(0)} \} = \{ \ddu{2}{0}, - \frac{\partial}{\partial x_{2,2}} - u_{1}^{(0)}\frac{\partial}{\partial x_{3}} \} = \ol{H_{1,2}^{(\0)}}.
\end{equation}
We use the second possibility, which amounts to prolonging the first input.

\subsubsection{Flatness by Pure Prolongation of the First Input}\label{ex1:input1:subsec}

Let us now apply theorem~\ref{cns0:thm} and algorithm~\ref{algo1} with $j_{2}=0$, \ie $g_{2}^{(j_{2})} \triangleq g_{2}^{(0)}= \frac{\partial}{\partial u_{2}^{(0)}}$ to determine if this system is flat by pure prolongation and compute $j_{1}\geq 1$.
Recall that we have set $g_{0}^{(l_{1},0)} = g_{0}^{(0)} + \sum_{p=0}^{l_{1}-1}u_{1}^{(p+1)}\ddu{1}{p}$.

\paragraph{$\bullet~ \mathbf{k=0.}$} We have $\Gamma_{0}^{(l_{1},0)}= \left\{ \frac{\partial}{\partial u_{1}^{(l_{1})}} \right\}$,
$\Delta_{0}^{(l_{1},0)} =  \left\{ \frac{\partial}{\partial u_{2}^{(0)}}\right\} = \ol{\Delta_{0}^{(l_{1},0)}}$ and $\left[ \Gamma_{0}^{(l_{1},0)}, \Delta_{0}^{(l_{1},0)}\right] \subset  \Delta_{0}^{(l_{1},0)}$ for all $l_{1}\geq 1$.

\paragraph{$\bullet~ \mathbf{k=1.}$} If $l_{1}\geq 2$,  
$$\Gamma_{1}^{(l_{1},0)} = \left\{ \frac{\partial}{\partial u_{1}^{(l_{1})}},  \frac{\partial}{\partial u_{1}^{(l_{1}-1)}} \right\}$$ 
and
$$\Delta_{1}^{(l_{1},0)} = \left\{ \ddu{2}{0}, \ad_{g_{0}^{(l_{1},0)}}\ddu{2}{0}\right\} = \left\{ \frac{\partial}{\partial u_{2}^{(0)}}, \frac{\partial}{\partial x_{2,2}} + u_{1}^{(0)}\frac{\partial}{\partial x_{3}}\right\} = \ol{\Delta_{1}^{(l_{1},0)}}.$$ 
Moreover,  $\left[ \Gamma_{1}^{(l_{1},0)}, \Delta_{1}^{(l_{1},0)}\right] = \{0\} \subset  \Delta_{1}^{(l_{1},0)}$.

Now, for $l_{1}=1$,
we have $\Gamma_{1}^{(1,0)}= \left\{ \frac{\partial}{\partial u_{1}^{(1)}} \right\}$ and
$\Delta_{1}^{(1,0)} = \left\{ \ddu{1}{0}, \ddu{2}{0}, \ad_{g_{0}^{(1,0)}}\ddu{2}{0} \right\} = \left\{ \frac{\partial}{\partial u_{1}^{(0)}},  \frac{\partial}{\partial u_{2}^{(0)}}, \frac{\partial}{\partial x_{2,2}} + u_{1}^{(0)}\frac{\partial}{\partial x_{3}} \right\}$ which is not involutive. 

Thus  $\sigma_{1,\Delta}(1)= 2$ and $\sigma_{1,\Gamma,\Delta}(1)=0$ which implies that 
$$j_{1} = \max_{k\geq 0} \sigma_{1,\Delta}(k) \wedge \sigma_{1,\Gamma,\Delta}(k) \geq \sigma_{1,\Delta}(1) \wedge \sigma_{1,\Gamma,\Delta}(1) = 2.$$

\paragraph{$\bullet~ \mathbf{k=2.}$} If  $l_{1}\geq 3$, we have 
$$\Gamma_{2}^{(l_{1},0)}= \left\{ \ddu{1}{l_{1}}, \ddu{1}{l_{1}-1} , \ddu{1}{l_{1}-2}\right\}
$$ 
and 
$$\begin{aligned}
\Delta_{2}^{(l_{1},0)} &= \left\{ \ddu{2}{0}, \ad_{g_{0}^{(l_{1},0)}}\ddu{2}{0},  \ad_{g_{0}^{(l_{1},0)}}^{2}\ddu{2}{0}\right\} \\
&= \left\{ \frac{\partial}{\partial u_{2}^{(0)}}, \frac{\partial}{\partial x_{2,2}} + u_{1}^{(0)}\frac{\partial}{\partial x_{3}}, \frac{\partial}{\partial x_{2,1}} - u_{1}^{(1)}\frac{\partial}{\partial x_{3}}\right\} = \ol{\Delta_{2}^{(l_{1},0)}}.
\end{aligned}$$

Moreover, it is readily verified that $\left[ \Gamma_{2}^{(l_{1},0)}, \Delta_{2}^{(l_{1},0)}\right] \subset  \Delta_{2}^{(l_{1},0)}$ only if $l_{1}\geq 4$, condition (ii) of theorem~\ref{cns0:thm} being violated if $l_{1} = 3$
and we have  $\sigma_{1,\Delta}(2)= 0$ and $\sigma_{1,\Gamma,\Delta}(2)=4$ which implies that $j_{1} = \max_{k\geq 0} \sigma_{1,\Delta}(k) \wedge \sigma_{1,\Gamma,\Delta}(k) \geq \max_{r=1,2}\sigma_{1,\Delta}(r) \wedge \sigma_{1,\Gamma,\Delta}(r)  \geq 4$.

\paragraph{$\bullet~ \mathbf{k=3.}$} Again, if  $l_{1}\geq 4$, we have:
$$\Gamma_{3}^{(l_{1},0)}= \left\{ \ddu{1}{l_{1}}, \ddu{1}{l_{1}-1} , \ddu{1}{l_{1}-2)}, \ddu{1}{l_{1}-3}\right\}$$ 
and 
$$
\begin{aligned}
\Delta_{3}^{(l_{1},0)} &= \left\{ \ddu{2}{0}, \ad_{g_{0}^{(l_{1},0)}}\ddu{2}{0},  \ad_{g_{0}^{(l_{1},0)}}^{2}\ddu{2}{0},  \ad_{g_{0}^{(l_{1},0)}}^{3}\ddu{2}{0} \right\} \\
&= \left\{ \frac{\partial}{\partial u_{2}^{(0)}}, \frac{\partial}{\partial x_{2,2}} + u_{1}^{(0)}\frac{\partial}{\partial x_{3}}, \frac{\partial}{\partial x_{2,1}} - u_{1}^{(1)}\frac{\partial}{\partial x_{3}}, u_{1}^{(2)}\frac{\partial}{\partial x_{3}}\right\} \\
&= \left\{ \frac{\partial}{\partial u_{2}^{(0)}}, \frac{\partial}{\partial x_{2,2}}, \frac{\partial}{\partial x_{2,1}}, \frac{\partial}{\partial x_{3}}\right\} = \ol{\Delta_{3}^{(l_{1},0)}}
\end{aligned}
$$
provided that $u_{1}^{(2)} \neq 0$.
We also indeed have $\left[ \Gamma_{3}^{(l_{1},0)}, \Delta_{3}^{(l_{1},0)}\right] \subset  \Delta_{3}^{(l_{1},0)}$ for $l_{1}\geq 4$ hence $\sigma_{1,\Gamma,\Delta}(3) = 4$ and $j_{1} \geq \max_{r=1,2,3}\sigma_{1,\Delta}(r) \wedge \sigma_{1,\Gamma,\Delta}(r)  \geq 4$.

\paragraph{$\bullet~ \mathbf{k = 4}$} If $l_{1}\geq 5$, we have $\Gamma_{4}^{(l_{1},0)} = \left\{ \frac{\partial}{\partial u_{1}^{(l_{1})}},  \ldots, \frac{\partial}{\partial u_{1}^{(l_{1}-4)}} \right\}$ and $\Delta_{4}^{(l_{1},0)}= \Delta_{3}^{(l_{1},0)}$ since $\ad_{g_{0}^{(l_{1},0)}}^{4}\ddu{2}{0}= -u_{1}^{(3)}\frac{\partial}{\partial x_{3}} \in \Delta_{3}^{(l_{1},0)}$. 

If now $l_{1}=4$, $\Gamma_{4}^{(4,0)} = \left\{ \frac{\partial}{\partial u_{1}^{(4)}},  \ldots, \frac{\partial}{\partial u_{1}^{(1)}} \right\}$
and
$$\begin{aligned}
\Delta_{4}^{(4,0)} &=  \left\{ \ddu{1}{0}, \ddu{2}{0}, \ad_{g_{0}^{(4,0)}}\ddu{2}{0}, \ldots,  \ad_{g_{0}^{(4,0)}}^{4}\ddu{2}{0} \right\} \\
&= \left\{ \frac{\partial}{\partial u_{1}^{(0)}}, \frac{\partial}{\partial u_{2}^{(0)}}, \frac{\partial}{\partial x_{2,2}}, \frac{\partial}{\partial x_{2,1}}, \frac{\partial}{\partial x_{3}} \right\} = \ol{\Delta_{4}^{(4,0)}}.
\end{aligned}
$$ 
We thus immediately get
$\left[ \Gamma_{4}^{(l_{1},0)}, \Delta_{4}^{(l_{1},0)}\right] \subset \Delta_{4}^{(l_{1},0)}$ for all $l_{1}\geq 4$.

\paragraph{$\bullet~ \mathbf{k\geq 5}$}
Finally, the reader may easily check that
$\Gamma_{k}^{(4,0)} = \Gamma_{4}^{(4,0)}$ for all $k\geq 5$ and 
$$\begin{aligned}
\Delta_{5}^{(4,0)} &=  \Delta_{4}^{(4,0)} + \left\{  \ad_{g_{0}^{(4,0)}}\ddu{1}{0} \right\} = \left\{ \frac{\partial}{\partial u_{1}^{(0)}}, \frac{\partial}{\partial x_{1,3}}, \frac{\partial}{\partial u_{2}^{(0)}}, \frac{\partial}{\partial x_{2,2}}, \frac{\partial}{\partial x_{2,1}}, \frac{\partial}{\partial x_{3}} \right\} \\
&= \ol{\Delta_{5}^{(4,0)}}\\
\Delta_{6}^{(4,0)} &= \Delta_{5}^{(4,0)} + \left\{  \ad_{g_{0}^{(4,0)}}^{2}\ddu{1}{0} \right\} =
\left\{ \frac{\partial}{\partial u_{1}^{(0)}}, \frac{\partial}{\partial x_{1,3}}, \frac{\partial}{\partial x_{1,2}}, \frac{\partial}{\partial u_{2}^{(0)}}, \frac{\partial}{\partial x_{2,2}}, \frac{\partial}{\partial x_{2,1}}, \frac{\partial}{\partial x_{3}} \right\} \\
&= \ol{\Delta_{6}^{(4,0)}}\\
\Delta_{7}^{(4,0)} &= \Delta_{6}^{(4,0)} + \left\{  \ad_{g_{0}^{(4,0)}}^{3}\ddu{1}{0} \right\} = \T \R^{8}.
\end{aligned}$$
We also indeed have
$\left[ \Gamma_{k}^{(4,0))}, \Delta_{k}^{(4,0))}\right] \subset  \Delta_{k}^{(4,0)}$ for all $k\geq 5$.

Hence, for all $k\geq 0$, the minimal $\mbf{j}$ is equal to $(4,0)$ and we conclude that system \eqref{chainsys:eq}, with the first input channel prolonged up to $j_{1}=4$, \ie controlled by $u_{1}^{(5)}$, is feedback linearizable. 

Note that the bounds \eqref{ineq:kstar-j} and \eqref{ineq-j:eq} are indeed satisfied. They read $ n+ j_{1}= 10 \geq k_{\star}^{(4,0)} = 7 \geq \frac{n + j_{1}}{2} \wedge j_{1}= 5$, but they are not tight. 

Accordingly, evaluating the bound on the number of integrators needed to linearize the system, proposed by \cite{SlT-scl96} for  $m=2$, we find $2n-3= 9$, and the one proposed by \cite{FrFo-ejc2005} gives $2n - \frac{1}{6}(8+24-14)= 2n-3=9$ as well.

\subsubsection{Flat output computation}
The prolonged system is now expressed in the state coordinates
$$\ol{x}^{(4,0)} \triangleq (x_{1,1}, x_{1,2}, x_{1,3}, x_{2,1}, x_{2,2}, x_{3}, u_{1}^{(0)}, u_{1}^{(1)}, u_{1}^{(2)}, u_{1}^{(3)}, u_{1}^{(4)},u_{2}^{(0)})$$ 
still with $x_{1,j} \triangleq x_{1}^{(j-1)}$, $j=1,2,3$, $x_{2,j} \triangleq x_{2}^{(j-1)}$, $j=1,2$ and $u_{i}^{(0)} = u_{i}$, $i=1,2$. We indeed still have $n=6$ but the prolonged state dimension is now equal to $n+m+\vert j\vert =12$ with the two input variables $u_{1}^{(5)}$ and $u_{2}^{(1)}$.

The prolonged system vector fields are
$$
g_{0}^{(4,0)} \triangleq g_{0}^{(0)} + \sum_{j=0}^{3} u_{1}^{(j+1)}\ddu{1}{j}, \qquad
g_{1}^{(4)} \triangleq \ddu{1}{4},\qquad
g_{2}^{(0)} \triangleq \ddu{2}{0}
$$
and the corresponding distributions are, according to \eqref{G-k-j:eq},
%
%
%
%
$$\begin{aligned}
G_{0}^{(4,0)} &= \left\{\ddu{1}{4}, \ddu{2}{0}\right\} =  \ol{G_{0}^{(4,0)}}, \\
G_{1}^{(4,0)} &= G_{0}^{(4,0)} +  \left\{\ddu{1}{3}, - \frac{\partial}{\partial x_{2,2}} - u_{1}^{(0)}\frac{\partial}{\partial x_{3}}\right\} =  \ol{G_{1}^{(4,0)}},\\
G_{2}^{(4,0)} &= G_{1}^{(4,0)} + \left\{\ddu{1}{2},  \frac{\partial}{\partial x_{2,1}} - u_{1}^{(1)}\frac{\partial}{\partial x_{3}}\right\} =  \ol{G_{2}^{(4,0)}},\\
G_{3}^{(4,0)} &= G_{2}^{(4,0)} + \left\{\ddu{1}{1}, - u_{1}^{(2)}\frac{\partial}{\partial x_{3}}\right\} = \ol{G_{3}^{(4,0)}},\\
G_{4}^{(4,0)} &= G_{3}^{(4,0)} + \left\{\ddu{1}{0}\right\} = \ol{G_{4}^{(4,0)}}, \\
G_{5}^{(4,0)} &= G_{4}^{(4,0)} + \left\{ \frac{\partial}{\partial x_{1,3}} \right\} = \ol{G_{5}^{(4,0)}},\\
G_{6}^{(4,0)} &= G_{5}^{(4,0)} + \left\{ \frac{\partial}{\partial x_{1,2}}\right\} = \ol{G_{6}^{(4,0)}}, \\
G_{7}^{(4,0)} &= G_{6}^{(4,0)} + \left\{ \frac{\partial}{\partial x_{1,1}}\right\} = \ol{G_{7}^{(4,0)}} = \T\R^{12}.
\end{aligned}$$
This confirms that the system \eqref{chainsys:eq} is flat by pure prolongation in any neighborhood excluding $u_{1}^{(2)}=0$, with $k^{(4,0)}_{\ast}= 7$. The reader may easily verify that $\rho_{k}^{(4,0)}= 2$ for $k= 0, \ldots, 3$ and $\rho_{k}^{(4,0)}= 1$ for $k= 4, \ldots, 7$, which yields $\kappa_{1}^{(4,0)}=8$ and $\kappa_{2}^{(4,0)}=4$.

The corresponding flat outputs are thus obtained by solving the set of P.D.E.'s
\begin{equation}\label{Gk4y:eq}
\begin{aligned}
&\left < G_{k}^{(4,0)}, dy_{1} \right > =0, \quad k= 0,\ldots, 6,  \qquad \left < G_{7}^{(4,0)}, dy_{1} \right > \neq 0\\
&\left < G_{k}^{(4,0)}, dy_{2} \right > =0, \quad k= 0,\ldots, 2, \qquad \left < G_{3}^{(4,0)}, dy_{2} \right > \neq 0
\end{aligned}
\end{equation}
whose solution is
\begin{equation}\label{ex1sol:eq}
y_{1}= x_{1} , \quad y_{2} =  x_{3} - x_{2,2}u_{1}^{(0)} + x_{2,1}u_{1}^{(1)} =  x_{3} - \dot{x}_{2}u_{1} + x_{2}\dot{u}_{1}.
\end{equation}

\begin{rem}
Using the first intialization \eqref{init1ex1:eq}, which amounts to prolonging the second input, we leave to the reader the verification that one obtains a prolongation of order 6, thus non minimal, with flat outputs given by \eqref{chain-f-out:eq}.
\end{rem}

\begin{rem}
In \cite[section 3.1.1]{MMR-ecc}, the authors consider a dual notion of minimality, called $r$-flatness, where $r$ is the minimal number over all possible flat outputs of the maximal number of derivatives of the inputs that appear in the flat outputs, \ie 
$$r= \min_{Y : ~\mathrm{flat~output}}~\max_{i=1,\ldots,m}\{s_{i} \mid \mbf{s}=(s_{1}, \ldots, s_{m}),~ y=Y(x,\ol{u}^{(\mbf{s})})\}.$$ They conjectured that $r$ should be equal to 1 in the present case (with their notations, $\alpha_{1}=3$, $\alpha_{2}=2$ and $\min(\alpha_{1},\alpha_{2})-1=1$). As the reader may easily verify, it is indeed minimal since the minimal $\mbf{j}$ is $(4,0)$ and moreover since, by the equations of the first line of \eqref{Gk4y:eq}, $y_{1}$ neither can depend on $\ol{u}_{1}^{(4)}$ nor on $u_{2}^{(0)}$, and, by the second line, $y_{2}$ cannot depend on $u_{1}^{(4)}, u_{1}^{(3)}, u_{1}^{(2)}$ but explicitly depends on $u_{1}^{(1)}$ by the definition of $G_{2}^{(4,0)}$.
\end{rem}

\subsection{4-dimensional Driftless Bilinear System \cite{MMR-ecc,MM-scl95,MM-cdc95}}\label{driftlessex:subsec}

\begin{equation}\label{driftless-sys:eq}
\begin{aligned}
\dot{x}_{1} &= u_{1}\\
\dot{x}_{2} &= x_{3}u_{1}\\
\dot{x}_{3} &= x_{4}u_{1}\\
\dot{x}_{4} &= u_{2}
\end{aligned}
\end{equation}
It is immediate to verify that this system is flat with flat output 
\begin{equation}\label{driftless-f-out:eq}
y_{1}= x_{1}, \quad y_{2} = x_{2}
\end{equation}
(see \cite{MM-scl95,MM-cdc95} and \cite[theorems 4 and 5]{MMR-ecc}) but not static feedback linearizable.

According to our formalism, we consider the state $(x_{1}, x_{2}, x_{3}, x_{4}, u_{1}^{(0)}, u_{2}^{(0)})$ of dimension 6, with $n=4$ and $m=2$, and the new inputs $(u_{1}^{(1)}, u_{2}^{(1)})$. The associated vector fields are
$$ g_{0}^{(0)}= u_{1}^{(0)} \left( \frac{\partial}{\partial x_{1}} + x_{3}\frac{\partial}{\partial x_{2}} + x_{4}\frac{\partial}{\partial x_{3}} \right) + u_{2}^{(0)}\frac{\partial}{\partial x_{4}} , \quad g_{1}^{(0)}= \ddu{1}{0}, \quad g_{2}^{(0)}= \ddu{2}{0}.$$
The reader may easily verify that $G_{1}^{(\0)}$ is not involutive and that there are two possible initializations:
$$H_{1,1}^{(\0)}= \{ g_{1}^{(0)}, \ad_{g_{0}^{(\0)}}g_{1}^{(0)} \} = \{ \ddu{1}{0},  \frac{\partial}{\partial x_{1}} + x_{3}\frac{\partial}{\partial x_{2}} + x_{4}\frac{\partial}{\partial x_{3}} \} = \ol{H_{1,1}^{(\0)}}$$
or 
$$H_{1,2}^{(\0)}= \{ g_{2}^{(0)}, \ad_{g_{0}^{(\0)}}g_{2}^{(0)} \} = \{ \ddu{2}{0},  \frac{\partial}{\partial x_{4}} \} = \ol{H_{1,2}^{(\0)}}.$$
We use the second one, which amounts to prolonging the first input $u_{1}^{(1)}$.

\subsubsection{the distributions $\Gamma_{k}^{(\mbf{l})}$ and $\Delta_{k}^{(\mbf{l})}$}
We set  $\mbf{l}= (l_{1},0)$, $l_{1}\geq 1$.
We thus have
$$g_{0}^{(l_{1},0)}= g_{0}^{(0)} + \sum_{p= 0}^{l_{1}-1}u_{1}^{(p+1)}\ddu{1}{p}, \quad 
g_{1}^{(l_{1})}=\frac{\partial}{\partial u_{1}^{(l_{1})}}, \quad \quad g_{2}^{(0)}= \ddu{2}{0}$$
and
$$\Gamma_{0}^{(l_{1},0)}= \left\{ \ddu{1}{l_{1}}\right\}, \quad \Delta_{0}^{(l_{1},0)} = \left\{ \ddu{2}{0} \right\}$$
for all $l_{1}\geq 1$.

\paragraph{$\bullet~ \mathbf{k=1.}$} If $l_{1}\geq 2$, we have 
$$\Gamma_{1}^{(l_{1},0)}= \left\{ \ddu{1}{l_{1}}, \ddu{1}{l_{1}-1}\right\}, \quad \Delta_{1}^{(l_{1},0)} = 
\Delta_{0}^{(l_{1},0)} + \{  \ad_{g_{0}^{(l_{1},0)}} \ddu{2}{0} \}
=\left\{ \ddu{2}{0}, \frac{\partial}{\partial x_{4}} \right\}$$
and if $l_{1} = 1$, 
$$\Gamma_{1}^{(1,0)} = \left\{ \ddu{1}{1} \right\}, \quad \Delta_{1}^{(1,0)} =  
\Delta_{0}^{(1,0)} + \{  \ddu{1}{0}, \ad_{g_{0}^{(1,0)}} \ddu{2}{0} \} =
\left\{ \ddu{1}{0}, \ddu{2}{0}, \frac{\partial}{\partial x_{4}} \right\}.$$
We thus have $\Delta_{1}^{(l_{1},0)} = \ol{\Delta_{1}^{(l_{1},0)}}$ and $\left[\Gamma_{1}^{(l_{1},0)}, \Delta_{1}^{(l_{1},0)}\right] = \{0\} \subset \Delta_{1}^{(l_{1},0)}$ for all $l_{1}\geq 1$, hence $\sigma_{1,\Gamma,\Delta}(1)= \sigma_{1,\Delta}(1) = 0$.

\paragraph{$\bullet~ \mathbf{k=2.}$} If $l_{1}\geq 3$, we have 
$$\begin{array}{c}
\Gamma_{2}^{(l_{1},0)}= \left\{ \ddu{1}{l_{1}}, \ddu{1}{l_{1}-1}, \ddu{1}{l_{1}-2}\right\}, \quad \Delta_{2}^{(l_{1},0)} = \left\{ \ddu{2}{0}, \frac{\partial}{\partial x_{4}}, \frac{\partial}{\partial x_{3}} \right\},\vspace{10pt} \\ 
\left[\Gamma_{2}^{(l_{1},0)}, \Delta_{2}^{(l_{1},0)} \right] = \{0\} 
\end{array}$$
and if $l_{1} = 2$, 
$$\begin{array}{c}
\Gamma_{2}^{(2,0)} = \left\{ \ddu{1}{2}, \ddu{1}{1} \right\}, \quad \Delta_{2}^{(2,0)} = \left\{ \ddu{1}{0}, \ddu{2}{0}, \frac{\partial}{\partial x_{4}}, \frac{\partial}{\partial x_{3}} \right\}, \vspace{10pt} \\ 
\left[\Gamma_{2}^{(2,0)}, \Delta_{2}^{(2,0)} \right] = \{0\}.\end{array}$$
Finally, if $l_{1} = 1$,
$$\begin{aligned}
\Gamma_{2}^{(1,0)} = \left\{ \ddu{1}{1} \right\}, \quad \Delta_{2}^{(1,0)} &= \left\{ \ddu{1}{0},  - \frac{\partial}{\partial x_{1}} - x_{3}\frac{\partial}{\partial x_{2}} - x_{4}\frac{\partial}{\partial x_{3}},
\ddu{2}{0}, \frac{\partial}{\partial x_{4}}, \frac{\partial}{\partial x_{3}} \right\} \\
&\neq \ol{ \Delta_{2}^{(1,0)}}.
\end{aligned}$$

We thus have $\Delta_{2}^{(l_{1},0)} = \ol{\Delta_{2}^{(l_{1},0)}}$ and $\left[\Gamma_{2}^{(l_{1},0)}, \Delta_{2}^{(l_{1},0)}\right] \subset  \Delta_{2}^{(l_{1},0)}$ for all $l_{1}\geq 2$ but, for $l_{1}=1$, $\Delta_{2}^{(1,0)}$ is not involutive. Therefore, $\sigma_{1,\Delta}(2) = 2$ and our search may be restricted to $l_{1}\geq 2$.

\paragraph{$\bullet~ \mathbf{k=3.}$} If $l_{1}\geq 4$, we have 
$$\Gamma_{3}^{(l_{1},0)}= \left\{ \ddu{1}{l_{1}}, \ddu{1}{l_{1}-1}, \ddu{1}{l_{1}-2}, \ddu{1}{l_{1}-3}\right\}, \quad \Delta_{3}^{(l_{1},0)} = \left\{ \ddu{2}{0}, \frac{\partial}{\partial x_{4}}, \frac{\partial}{\partial x_{3}}, \frac{\partial}{\partial x_{2}}  \right\}.$$

If $l_{1} = 3$, 
$$\Gamma_{3}^{(3,0)}= \left\{ \ddu{1}{3}, \ddu{1}{2}, \ddu{1}{1} \right\}, \quad \Delta_{3}^{(3,0)} = \left\{ \ddu{1}{0}, \ddu{2}{0}, \frac{\partial}{\partial x_{4}}, \frac{\partial}{\partial x_{3}}, \frac{\partial}{\partial x_{2}}  \right\}.$$

If $l_{1} = 2$, 
$$\Gamma_{3}^{(2,0)}= \left\{ \ddu{1}{2}, \ddu{1}{1} \right\}, \quad \Delta_{3}^{(2,0)} = \left\{ \ddu{1}{0}, \ddu{2}{0}, \frac{\partial}{\partial x_{4}}, \frac{\partial}{\partial x_{3}}, \frac{\partial}{\partial x_{2}},  \frac{\partial}{\partial x_{1}}  \right\}= \T\R^{6}$$

Therefore, $\Delta_{k}^{(2,0)}  = \ol{\Delta_{k}^{(2,0)}}$ and $\left[\Gamma_{k}^{(2,0)},\Delta_{k}^{(2,0)} \right] \subset \Delta_{k}^{(2,0)} $ for all $k\geq 0$ and $\Delta_{3}^{(2,0)} = \T\R^{6}$.

We conclude that the conditions of theorem~\ref{cns0:thm} hold true whenever $j_{1}\geq 2$, which proves that system~\eqref{driftless-sys:eq} with the pure prolongation of order $\mbf{j}=(2,0)$ is feedback linearizable.

The reader may easily check that, using the first initialization $H_{1,1}^{(\0)}$, since $\Delta_{2}^{(0, l_{2})}$ is not involutive for all $l_{2} \geq 3$, no pure prolongation of the second input channel, $u_{2}^{(1)}$, can lead to the linearizability conditions. Therefore the minimal prolongation is $\mbf{j}=(2,0)$.

\subsubsection{Flat output computation}

The prolonged state is now $\ol{x}^{(2,0)} \triangleq (x_{1}, x_{2}, x_{3}, x_{4}, u_{1}^{(0)}, u_{1}^{(1)}, u_{1}^{(2)}, u_{2}^{(0)})$ of dimension 8, and the new inputs are $(u_{1}^{(3)}, u_{2}^{(1)})$. 
$$
G_{0}^{(2,0)} = \left\{\ddu{1}{2}, \ddu{2}{0}\right\} = \ol{G_{0}^{(2,0)}}, 
$$
$$
G_{1}^{(2,0)} = \left\{\ddu{1}{2}, \ddu{1}{1}, \ddu{2}{0}, \frac{\partial}{\partial x_{4}}\right\} = \ol{G_{1}^{(2,0)}},
$$
$$
G_{2}^{(2,0)} = \left\{\ddu{1}{2}, \ddu{1}{1}, \ddu{1}{0}, \ddu{2}{0},  \frac{\partial}{\partial x_{4}}, \frac{\partial}{\partial x_{3}}\right\} =  \ol{G_{2}^{(2,0)}},
$$
$$\begin{aligned}
G_{3}^{(2,0)} &= \left\{\ddu{1}{2}, \ddu{1}{1}, \ddu{1}{0}, \frac{\partial}{\partial x_{1}} + x_{3}\frac{\partial}{\partial x_{2}} + x_{4}\frac{\partial}{\partial x_{3}}, \ddu{2}{0}, \frac{\partial}{\partial x_{4}}, \frac{\partial}{\partial x_{3}}, \frac{\partial}{\partial x_{2}}\right\}\\ 
&= \ol{G_{3}^{(2,0)}}  = \T\R^{8}
\end{aligned}
$$
hence the feedback linearizability of the purely prolonged system with 
$\rho_{0}^{(2,0)} =\rho_{1}^{(2,0)}= \rho_{2}^{(2,0)}=\rho_{3}^{(2,0)}=2$ and $\kappa_{1}^{(2,0)}=\kappa_{2}^{(2,0)}=4$.

Flat outputs $(y_{1}, y_{2})$ are locally non trivial solutions of the system~\eqref{PDEs2-0:eq}, \ie:
\begin{equation}\label{PDEs2-0:eq}
\left< G_{k}^{(2,0)}, dy_{i} \right> = 0, \; k= 0,1, 2, \quad \mathrm{with~}                           
\quad \left< G_{3}^{(2,0)}, dy_{i} \right> \neq 0, \quad i=1,2.
\end{equation}
It is immediate to verify that 
$$y_{1} = x_{1}, \qquad y_{2} = x_{2}$$ 
is a solution of \eqref{PDEs2-0:eq} and that the mapping
$$\ol{x}^{(2,0)} \mapsto \left(y_{1}, \ldots, y_{1}^{(3)}, y_{2}, \ldots, y_{2}^{(3)}\right)$$
is a local diffeomorphism.

\subsection{An Example from \cite{CLM}}\label{5:3:ex}
In our formalism, considering the inputs $(u_{1}, u_{2}) \triangleq (u_{1}^{(0)}, u_{2}^{(0)})$ as part of the state, with $n=4$ and $m=2$, this example from \cite[Example 2]{CLM} reads: 
\begin{equation}\label{clm-ex:eq}
\begin{aligned}
&\dot{x}_{1} = x_{2} + x_{3}u_{2}^{(0)}\\
&\dot{x}_{2} = x_{3} + x_{1}u_{2}^{(0)}\\
&\dot{x}_{3} = u_{1}^{(0)} + x_{2}u_{2}^{(0)}\\
&\dot{x}_{4} = u_{2}^{(0)}\\
&\dot{u}_{1}^{(0)} = u_{1}^{(1)}\\
&\dot{u}_{2}^{(0)} = u_{2}^{(1)}.
\end{aligned}
\end{equation}

It is shown in \cite{CLM} that this sytem does not satisfy the sufficient, but otherwise not necessary, condition for dynamic linearization of Theorem 4.2 of this paper. Nevertheless, it satisfies the algorithm of \cite[section 3.1]{BC-scl2004}, that constitutes a sufficient condition for flatness by pure prolongation, without proof of minimality of the obtained prolongation. We show here that it is linearizable by pure prolongation by application of our algorithm, thus providing the minimal prolongation.

The non prolonged vctor fields are 
$$
\begin{aligned}
&g_{0}^{(\0)} = (x_{2} + x_{3}u_{2}^{(0)})\frac{\partial}{\partial x_{1}} 
+  (x_{3} + x_{1}u_{2}^{(0)})\frac{\partial}{\partial x_{2}}
+ (u_{1}^{(0)} + x_{2}u_{2}^{(0)})\frac{\partial}{\partial x_{3}}
+ u_{2}^{(0)}\frac{\partial}{\partial x_{4}} \\
&
g_{1}^{(\0)} = \ddu{1}{0}, \qquad g_{2}^{(0)} = \ddu{2}{0}.
\end{aligned}
$$
and it is easily seen that 
$$G_{1}^{(\0)} = \{  \ddu{1}{0}, \ddu{2}{0}, \ddxx{3},  - x_{3} \ddxx{1} - x_{1}\ddxx{2} - x_{2}\ddxx{3} - \ddxx{4} \} \neq \ol{G_{1}^{(\0)}}$$
and that the subdistributions
$$H_{1,1}^{0} = \{ g_{1}^{(\0)}, \ad_{g_{0}^{(\0)}}g_{1}^{(\0)} \} = \{ \ddu{1}{0}, \ddxx{3} \} = \ol{H_{1,1}^{0}}$$
or 
$$H_{1,2}^{0} = \{ g_{2}^{(\0)}, \ad_{g_{0}^{(\0)}}g_{2}^{(\0)} \} = \{ \ddu{2}{0}, - x_{3} \ddxx{1} - x_{1}\ddxx{2} - x_{2}\ddxx{3} - \ddxx{4} \} = \ol{H_{1,2}^{0}}$$
can be taken as possible initializations.

Choosing $H_{1,1}^{0}$ amounts to prolonging the second input at an arbitrary order $l_{2}\geq 1$ and set $\bl= (0, l_{2})$. For $l_{2}\geq 1$, we denote, as before,
$$
\begin{aligned}
&g_{0}^{(0,l_{2})} = (x_{2} + x_{3}u_{2}^{(0)})\frac{\partial}{\partial x_{1}} 
+  (x_{3} + x_{1}u_{2}^{(0)})\frac{\partial}{\partial x_{2}}
+ (u_{1}^{(0)} + x_{2}u_{2}^{(0)})\frac{\partial}{\partial x_{3}}\\
&\hspace{6cm} 
+ u_{2}^{(0)}\frac{\partial}{\partial x_{4}} + \sum_{p= 0}^{l_{2}-1} u_{2}^{(p+1)}\ddu{2}{p}\\
&
g_{1}^{(\0)} = \ddu{1}{0}, \qquad g_{2}^{(l_{2})} = \ddu{2}{l_{2}}
\end{aligned}
$$
and we indeed have
$$\Gamma_{0}^{(0,l_{2})}=\left\{ \ddu{2}{l_{2}}\right\}, \quad \Delta_{0}^{(0,l_{2})} = \left\{ \ddu{1}{0}\right\} = \ol{\Delta_{0}^{(0,l_{2})}}, \quad \left[ \Gamma_{0}^{(0,l_{2})}, \Delta_{0}^{(0,l_{2})}\right] \subset \Delta_{0}^{(0,l_{2})}, \quad \forall l_{2}\geq 1.$$

\paragraph{$\bullet~ \mathbf{k=1.}$} For all $l_{2}\geq 2$, $\ds \ad_{g_{0}^{(0,l_{2})}}\ddu{1}{0} = - \frac{\partial}{\partial x_{3}}$ and
$$\begin{array}{c}
\ds \Gamma_{1}^{(0,l_{2})}=\left\{ \ddu{2}{l_{2}}, \ddu{2}{l_{2}-1}\right\}, \qquad \Delta_{1}^{(0,l_{2})} = \left\{ \ddu{1}{(0)},
-\frac{\partial}{\partial x_{3}}\right\} = \ol{\Delta_{1}^{(0,l_{2})}}, \vspace{1em}\\
\ds \left[ \Gamma_{1}^{(0,l_{2})}, \Delta_{1}^{(0,l_{2})}\right] \subset \Delta_{1}^{(0,l_{2})}.
\end{array}$$
For $l_{2}=1$: 
$$\begin{array}{c}
\ds \Gamma_{1}^{(0,1)}=\left\{ \ddu{2}{1} \right\}, \quad \Delta_{1}^{(0,1)} = \left\{ \ddu{2}{0}, \ddu{1}{0},
-\frac{\partial}{\partial x_{3}}\right\} = \ol{\Delta_{1}^{(0,1)}}, \vspace{1em}\\ 
\ds \left[ \Gamma_{1}^{(0,1)}, \Delta_{1}^{(0,1)}\right] \subset \Delta_{1}^{(0,1)}.
\end{array}$$

\paragraph{$\bullet~ \mathbf{k=2.}$} For all $l_{2}\geq 3$, $\ds \ad_{g_{0}^{(0,l_{2})}}^{2}\ddu{1}{0} = u_{2}^{(0)} \frac{\partial}{\partial x_{1}} +  \frac{\partial}{\partial x_{2}}$ and 
$$\begin{array}{c}
\ds \Gamma_{2}^{(0,l_{2})}=\left\{ \ddu{2}{l_{2}}, \ddu{2}{l_{2}-1}, \ddu{2}{l_{2}-2}\right\}, \vspace{1em}\\
\Delta_{2}^{(0,l_{2})} = \left\{ \ddu{1}{0},-\frac{\partial}{\partial x_{3}}, u_{2}^{(0)} \frac{\partial}{\partial x_{1}} +  \frac{\partial}{\partial x_{2}}\right\} = \ol{\Delta_{2}^{(0,l_{2})}}, \vspace{1em}\\
\ds  \left[ \Gamma_{2}^{(0,l_{2})}, \Delta_{2}^{(\0,l_{2})}\right] \subset \Delta_{2}^{(0,l_{2})}.
\end{array}$$
But for $l_{2}=2$: 
$$\Gamma_{2}^{(0,2)}=\left\{ \ddu{2}{2} , \ddu{2}{1}\right\}, \quad \Delta_{2}^{(0,2)} = \left\{ \ddu{2}{0}, \ddu{1}{0},
-\frac{\partial}{\partial x_{3}}, u_{2}^{(0)} \frac{\partial}{\partial x_{1}} +  \frac{\partial}{\partial x_{2}}\right\} \neq \ol{\Delta_{2}^{(0,2)}}$$
therefore, $\sigma_{2,\Delta}(2)=2$ and we must exclude $j_{2}=2$.

\paragraph{$\bullet~ \mathbf{k=3.}$} For all $l_{2} \geq 4$,  $\ds \ad_{g_{0}^{(0,l_{2})}}^{3} \ddu{1}{0} = (u_{2}^{(1)} - 1)\frac{\partial}{\partial x_{1}} - \left(u_{2}^{(0)}\right)^{2}\frac{\partial}{\partial x_{2}} - u_{2}^{(0)} \frac{\partial}{\partial x_{3}}$ and, if we exclude the points where $u_{2}^{(0)}=0$ and $u_{2}^{(1)} =1$,
$$\Gamma_{3}^{(0,l_{2})}=\left\{ \ddu{2}{l_{2}}, \ddu{2}{l_{2}-1}, \ddu{2}{l_{2}-2}, \ddu{2}{l_{2}-3}\right\},$$
$$\begin{aligned}
\Delta_{3}^{(\0,l_{2})} &= \left\{ \ddu{1}{0},-\frac{\partial}{\partial x_{3}}, u_{2}^{(0)} \frac{\partial}{\partial x_{1}} +  \frac{\partial}{\partial x_{2}}, (u_{2}^{(1)} - 1)\frac{\partial}{\partial x_{1}} + \left(u_{2}^{(0)}\right)^{2}\frac{\partial}{\partial x_{2}} - u_{2}^{(0)}\frac{\partial}{\partial x_{3}}\right\}\\
&= \left\{ \ddu{1}{(0)}, \frac{\partial}{\partial x_{3}}, \frac{\partial}{\partial x_{2}},  \frac{\partial}{\partial x_{1}} \right\} = \ol{\Delta_{3}^{(0,l_{2})}},
\end{aligned}$$
$$
\left[ \Gamma_{3}^{(0,l_{2})}, \Delta_{3}^{(0,l_{2})}\right] \subset \Delta_{3}^{(0,l_{2})}.
$$
The reader may then easily check that the same holds for $l_{2} = 3$:
$$\Gamma_{3}^{(0,3)}=\left\{ \ddu{2}{3}, \ddu{2}{2}, \ddu{2}{1}\right\},$$
$$\Delta_{3}^{(0,3)} = \left\{ \ddu{2}{0}, \ddu{1}{0}, \frac{\partial}{\partial x_{3}}, \frac{\partial}{\partial x_{2}},  \frac{\partial}{\partial x_{1}} \right\} = \ol{\Delta_{3}^{(0,3)}}.$$

\paragraph{$\bullet~ \mathbf{For~all~k\geq 4.}$} Since $\ad_{g_{0}^{(0,l_{2})}}^{k}\ddu{1}{0}$ and $\ad_{g_{0}^{(0,l_{2})}}^{k-3}\ddu{2}{0}$ are linear combinations of $\frac{\partial}{\partial x_{1}}$, $\frac{\partial}{\partial x_{2}}$ and $\frac{\partial}{\partial x_{3}}$ only, we have, for all $l_{2}\geq 4$:
$$
\Delta_{k}^{(0,l_{2})} = \ol{\Delta_{k}^{(0,l_{2})}}, \quad \left[ \Gamma_{k}^{(0,l_{2})}, \Delta_{k}^{(0,l_{2})}\right] \subset \Delta_{k}^{(0,l_{2})},
$$
and, for $l_{2}=3$, using the fact that 
$\ad_{g_{0}^{(0,3)}}\ddu{2}{0} = - x_{3} \frac{\partial}{\partial x_{1}} - x_{1}\frac{\partial}{\partial x_{2}} - x_{2}\frac{\partial}{\partial x_{3}} - \frac{\partial}{\partial x_{4}}$, we have
$$\Gamma_{4}^{(0,3)}=\left\{ \ddu{2}{3}, \ddu{2}{2}, \ddu{2}{1}\right\},$$
$$\begin{aligned}
\Delta_{4}^{(0,3)} &= \left\{ \ddu{2}{0}, - x_{3} \frac{\partial}{\partial x_{1}} - x_{1}\frac{\partial}{\partial x_{2}} - x_{2}\frac{\partial}{\partial x_{3}} - \frac{\partial}{\partial x_{4}}, \ddu{1}{0}, \frac{\partial}{\partial x_{3}}, \frac{\partial}{\partial x_{2}},  \frac{\partial}{\partial x_{1}} \right\} \\
&= \left\{ \ddu{2}{0}, \ddu{1}{0}, \frac{\partial}{\partial x_{4}}, \frac{\partial}{\partial x_{3}}, \frac{\partial}{\partial x_{2}},  \frac{\partial}{\partial x_{1}} \right\} = \ol{\Delta_{4}^{(0,3)}} = \T\R^{6}
\end{aligned}$$
and
$$\Delta_{k}^{(0,3)} = \T\R^{6}\quad \forall k\geq 4,$$
hence $j_{2} = 3$ and $k_{\ast}^{(0,3)}=4$.

We conclude that the conditions of theorem~\ref{cns0:thm} are satisfied for all $k$ provided that $j_{2} = 3$.

On the contrary, initializing the algorithm by $H_{1,2}^{0}$, which amounts to prolonging the first input $u_{1}^{(1)}$, we may easily check that, for all $l_{1}\geq 3$,
$$\Delta_{2}^{(l_{1},0)}= \left\{ \ddu{2}{0}, - x_{3} \frac{\partial}{\partial x_{1}} - x_{1}\frac{\partial}{\partial x_{2}} - x_{2}\frac{\partial}{\partial x_{3}} - \frac{\partial}{\partial x_{4}}, (x_{1} - u_{1}^{(0)}) \frac{\partial}{\partial x_{1}} - x_{3} \frac{\partial}{\partial x_{3}}\right\}$$ 
is not involutive, thus contradicting condition (i) of Theorem~\ref{cns0:thm}, which proves that the minimal prolongation of the second input is equal to 3.

Let us finally give the construction of the flat output and prolonged state diffeomorphism.
The prolonged state is $(x_{1}, x_{2}, x_{3}, x_{4}, u_{1}^{(0)}, u_{2}^{(0)}, u_{2}^{(1)}, u_{2}^{(2)}, u_{2}^{(3)})$ of dimension $9 = n+ m + \mid \mbf{j} \mid$. 

We get
$$
\begin{aligned}
&G_{0}^{(0,3)} = \left\{\ddu{1}{0}, \ddu{2}{3}\right\} = \ol{G_{0}^{(0,3)}}, \quad
G_{1}^{(0,3)} =  \left\{\frac{\partial}{\partial x_{3}}, \ddu{2}{2}\right\} \oplus G_{0}^{(0,3)} = \ol{G_{1}^{(0,3)}},\\
&G_{2}^{(0,3)} =  \left\{u_{2}^{(0)} \frac{\partial}{\partial x_{1}} +  \frac{\partial}{\partial x_{2}}, \ddu{2}{1}\right\} \oplus G_{1}^{(0,3)} = \ol{G_{2}^{(0,3)}}, \\
&G_{3}^{(0,3)} =  \left\{(u_{2}^{(1)} -1) \frac{\partial}{\partial x_{1}} - u_{2}^{(0)} \frac{\partial}{\partial x_{3}},  \ddu{2}{0}\right\} \oplus G_{2}^{(0,3)} = \ol{G_{3}^{(0,3)}}, \\
&G_{4}^{(0,3)}  \left\{- x_{3} \frac{\partial}{\partial x_{1}} - x_{1}\frac{\partial}{\partial x_{2}} - x_{2}\frac{\partial}{\partial x_{3}} - \frac{\partial}{\partial x_{4}}\right\} \oplus G_{4}^{(0,3)}  = \T\R^{9}
\end{aligned}
$$
The Brunovsk\'{y}'s controllability indices are $\kappa_{1}^{(0,3)}=5$ and $\kappa_{2}^{(0,3)}=4$ and the system of PDE's that the flat outputs must satisfy is :
$$\begin{aligned}
&\left < G_{k}^{(0,3)}, dy_{1} \right > =0, \quad k= 0,\ldots, 3, \qquad \left < G_{4}^{(0,3)}, dy_{1} \right > \neq 0\\
&\left < G_{k}^{(0,3)}, dy_{2} \right > =0, \quad k= 0,\ldots, 2,  \qquad \left < G_{3}^{(0,3)}, dy_{2} \right > \neq 0.
\end{aligned}$$
Its solution is given by $y_{1}=x_{4}$, $y_{2}= x_{1}-u_{2}^{(0)}x_{2}$.

\subsection{The Pendulum Example \cite[section II. C]{FLMR-ieee}}\label{pend:ex}
This model of pendulum in the vertical plane has been studied in \cite[section II. C]{FLMR-ieee}, \cite[section 6.2.3]{L_book},\cite[section 5.3]{L-AAECC} where it is shown to be flat. We prove here that it is not flat by pure prolongation. 

Though naturally control-affine, it is presented here in its prolonged form \eqref{control-aff-sys:eq}:
\begin{equation}\label{pend-sys:eq}
\begin{array}{cclcccl}
\ds \dot{x}_1 & = &\ds x_2 &~\quad&\ds  \dot{x}_2 & = &\ds u_1^{(0)}\\
\ds \dot{y}_1 & = &\ds y_2 &~\quad&\ds \dot{y}_2 & = &\ds u_2^{(0)}\\
\ds \dot{\theta}_1 & = &\ds \theta_2 &~\quad&\ds \dot{\theta}_2 & = &\ds -\frac{u_{1}^{(0)}}{\ee}\cos\theta_1 + \frac{u_{2}^{(0)} +1}{\ee}\sin\theta_1\\
\ds \dot{u}_1^{(0)} & = &\ds u_1^{(1)} &~\quad& \ds \dot{u}_2^{(0)} & = &\ds u_2^{(1)}.
\end{array}
\end{equation}

The state is 
$(x_1, x_2, y_1, y_2, \theta_1, \theta_2, u_1^{(0)}, u_2^{(0)})$, of dimension $n+m=6+2=8$. The associated non prolonged vector fields are 
\begin{equation}\label{pend-vecfield:eq}
\begin{aligned}
g_{0}^{(0)} &= x_2\frac{\partial}{\partial x_{1}} + y_2\frac{\partial}{\partial y_{1}} + \theta_2\frac{\partial}{\partial \theta_{1}}
+ \frac{1}{\ee}\sin\theta_1\frac{\partial}{\partial \theta_{2}} \\
&+ u_{1}^{(0)}\left( \frac{\partial}{\partial x_{2}} - \frac{1}{\ee}\cos\theta_{1}\frac{\partial}{\partial \theta_{2}}\right) + u_{2}^{(0)}\left(\frac{\partial}{\partial y_{2}} +  \frac{1}{\ee}\sin\theta_{1}\frac{\partial}{\partial \theta_{2}}\right) \\
g_{1}^{(0)} &= \ddu{1}{0} \qquad g_{2}^{(0)} = \ddu{2}{0}
\end{aligned}
\end{equation}
We have
$$G_{0}^{(\0)} = \{ g_{1}^{(0)}, g_{2}^{(0)} \} = \{ \ddu{1}{0}, \ddu{2}{0} \}$$
and 
$$\begin{aligned}
G_{1}^{(\0)} &= \{ g_{1}^{(0)}, g_{2}^{(0)}, \ad_{g_{0}^{(\0)}}g_{1}^{0}, \ad_{g_{0}^{(\0)}}g_{2}^{0} \}\\ 
&= \{ \ddu{1}{0}, \ddu{2}{0}, -\ddxx{2} +\frac{1}{\ee}\cos\theta_{1}\ddth{2}, -\ddyy{2} -\frac{1}{\ee}\sin\theta_{1}\ddth{2} \} = \ol{G_{1}^{(\0)}}
\end{aligned}$$
but
$$\begin{aligned}
G_{2}^{(\0)} &= G_{1}^{(\0)} + \{ \ddxx{1} - \frac{1}{\ee}\cos\theta_{1}\ddth{1} -\frac{\theta_{2}}{\ee}\sin\theta_{1}\ddth{2}, \ddyy{1} + \frac{1}{\ee}\sin\theta_{1}\ddth{1} -\frac{\theta_{2}}{\ee}\cos\theta_{1}\ddth{2} \}\\
&\neq \ol{G_{2}^{(\0)}}.
\end{aligned}$$

As before, we may initialize the pure prolongation algorithm by choosing
$$H_{1,1}^{0} = \{  g_{1}^{(0)}, \ad_{g_{0}^{(\0)}}g_{1}^{0} \} = \{ \ddu{1}{0},  -\ddxx{2} + \frac{1}{\ee}\cos\theta_{1}\ddth{2} \} = \ol{H_{1,1}^{0}}$$
which amounts to prolonging $u_{2}^{(1)}$, or
$$H_{1,2}^{0} = \{  g_{2}^{(0)}, \ad_{g_{0}^{(\0)}}g_{2}^{0} \} = \{ \ddu{2}{0},  -\ddyy{2} - \frac{1}{\ee}\sin\theta_{1}\ddth{2} \} = \ol{H_{1,2}^{0}}$$
which amounts to prolonging $u_{1}^{(1)}$.

It is noticeable that the inputs $u_{1}^{(1)}$ and $u_{2}^{(1)}$ play a symmetric role . Thus, one may choose indifferently one of them as the non prolonged input. Let us choose $u_{1}^{(1)}$ as non prolonged input, with the initialization $H_{1,1}^{0}$. Thus, the vector fields associated to a prolongation of length $l_{2}$ on the second input  read:
\begin{equation}\label{pend-vecfield-2prol:eq}
\begin{aligned}
g_{0}^{(0,l_{2})} &= x_2\frac{\partial}{\partial x_{1}} + y_2\frac{\partial}{\partial y_{1}} + \theta_2\frac{\partial}{\partial \theta_{1}} + \frac{1}{\ee}\sin\theta_1\frac{\partial}{\partial \theta_{2}} + u_{1}^{(0)}\left( \frac{\partial}{\partial x_{2}} - \frac{1}{\ee}\cos\theta_{1}\frac{\partial}{\partial \theta_{2}}\right)  \\
&+ u_{2}^{(0)}\left(\frac{\partial}{\partial y_{2}} +  \frac{1}{\ee}\sin\theta_{1}\frac{\partial}{\partial \theta_{2}}\right) + \sum_{p= 0}^{l_{2}-1} u_{2}^{(p+1)}\ddu{2}{p} \\
g_{1}^{(0)} &= \ddu{1}{0} \qquad g_{2}^{(l_{2})} = \ddu{2}{l_{2}}.
\end{aligned}
\end{equation}

\paragraph{$\bullet~ \mathbf{k=1.}$} For all $l_{2}\geq 2$, 
$$\Gamma_{1}^{(0,l_{2})} = \{\ddu{2}{l_{2}}, \ddu{2}{l_{2}-1}\}, \quad 
\Delta_{1}^{(0,l_{2})} = \{\ddu{1}{0}, - \frac{\partial}{\partial x_{2}} + \frac{1}{\ee}\cos\theta_{1}\frac{\partial}{\partial \theta_{2}}\},$$ 
and if $l_{2}=1$, 
$$\Gamma_{1}^{(0,1)} = \{\ddu{2}{1}\}, \quad
\Delta_{1}^{(0,l_{2})} = \{\ddu{1}{0}, -\frac{\partial}{\partial x_{2}} + \frac{1}{\ee}\cos\theta_{1}\frac{\partial}{\partial \theta_{2}}, \ddu{2}{0}\}.$$ 
Thus, we have $\ol{\Delta_{1}^{(0,l_{2})}} = \Delta_{1}^{(0,l_{2})}$ and $[\Gamma_{1}^{(0,l_{2})}, \Delta_{1}^{(0,l_{2})}] \subset \Delta_{1}^{(0,l_{2})}$  for all $l_{2}\geq 1$.

\bigskip

\paragraph{$\bullet~ \mathbf{k=2.}$} For all $l_{2}\geq 3$, 
$$\begin{aligned}
\Gamma_{2}^{(0,l_{2})} &= \{\ddu{2}{l_{2}}, \ddu{2}{l_{2}-1}, \ddu{2}{l_{2}-2}\}, \\
\Delta_{2}^{(0,l_{2})} &= \{\ddu{1}{0}, - \frac{\partial}{\partial x_{2}} + \frac{1}{\ee}\cos\theta_{1}\frac{\partial}{\partial \theta_{2}}, \frac{\partial}{\partial x_{1}} - \frac{1}{\ee}\cos\theta_{1}\frac{\partial}{\partial \theta_{1}} - \frac{1}{\ee}\theta_{2}\sin\theta_{1}\frac{\partial}{\partial \theta_{2}}\}.
\end{aligned}$$ 
If $l_{2}=2$, 
$$\begin{aligned}
\Gamma_{2}^{(0,2)} &= \{\ddu{2}{2}, \ddu{2}{1}\},\\ 
\Delta_{2}^{(0,2)} &= \{\ddu{1}{0}, - \frac{\partial}{\partial x_{2}} + \frac{1}{\ee}\cos\theta_{1}\frac{\partial}{\partial \theta_{2}}, \frac{\partial}{\partial x_{1}} - \frac{1}{\ee}\cos\theta_{1}\frac{\partial}{\partial \theta_{1}} - \frac{1}{\ee}\theta_{2}\sin\theta_{1}\frac{\partial}{\partial \theta_{2}}, \ddu{2}{0}\}.\end{aligned}$$ 
If $l_{2}=1$, 
$$\begin{aligned}
\Gamma_{2}^{(0,1)} &= \{\ddu{2}{1}\},\\ 
\Delta_{2}^{(0,1)} &= \{\ddu{1}{0}, - \frac{\partial}{\partial x_{2}} + \frac{1}{\ee}\cos\theta_{1}\frac{\partial}{\partial \theta_{2}}, \frac{\partial}{\partial x_{1}} - \frac{1}{\ee}\cos\theta_{1}\frac{\partial}{\partial \theta_{1}} - \frac{1}{\ee}\theta_{2}\sin\theta_{1}\frac{\partial}{\partial \theta_{2}}, \\
&\hspace{4cm}\ddu{2}{0}, -\frac{\partial}{\partial y_{2}} -  \frac{1}{\ee}\sin\theta_{1}\frac{\partial}{\partial \theta_{2}}\}.
\end{aligned}$$
Again, we have $[\Gamma_{2}^{(0,l_{2})}, \Delta_{2}^{(0,l_{2})}] \subset \Delta_{2}^{(0,l_{2})}$ for all $l_{2}\geq 1$, but $\ol{\Delta_{2}^{(0,l_{2})}} \neq \Delta_{2}^{(0,l_{2})}$ for all $l_{2}\geq 1$ since, \eg
$$[ - \frac{\partial}{\partial x_{2}} + \frac{1}{\ee}\cos\theta_{1}\frac{\partial}{\partial \theta_{2}}, \frac{\partial}{\partial x_{1}} - \frac{1}{\ee}\cos\theta_{1}\frac{\partial}{\partial \theta_{1}} - \frac{1}{\ee}\theta_{2}\sin\theta_{1}\frac{\partial}{\partial \theta_{2}}] =  - \frac{1}{\ee^{2}} \sin2\theta_{1}\frac{\partial}{\partial \theta_{2}} \not \in \Delta_{2}^{(0,l_{2})}$$ 
for all $l_{2} \geq 1$. 
Changing the non prolonged input $u_{1}^{(0)}$ in $u_{2}^{(0)}$, as previously announced, a similar calculation, left to the reader, shows that $\Delta_{2}^{(l_{1},0)}$ is not involutive for all $l_{1}\geq 1$. Thus,  according to the first item of theorem~\ref{alg-proof:thm}, system \eqref{pend-sys:eq} is not flat by pure prolongation, though differentially flat, as shown in \cite[section C]{FLMR-ieee}.

\subsection{An example with 3 inputs \cite{KLO_ijrnc20}}\label{3in:ex}

\begin{equation}\label{ex2-sys:eq}
\begin{array}{ccl}
\dot{x}_{1} & = & u_{1}^{(0)} \\
\dot{x}_{2} & = & x_{3} u_{1}^{(0)} \\ 
\dot{x}_{3} & = & x_{4} u_{1}^{(0)} + x_{1} u_{3}^{(0)} \\
\dot{x}_{4} & = & u_{2}^{(0)}.
\end{array}
\end{equation}
The state is $(x_{1}, \ldots, x_{4}, u_{1}^{(0)}, u_{2}^{(0)}, u_{3}^{(0)})$ (of dimension $n+m= 4+3=7$) and the control inputs are $(u_{1}^{(1)}, u_{2}^{(1)}, u_{3}^{(1)})$. The reader may easily see that $G_{1}^{(\0)}$ is not involutive and thus that this system is not static feedback linearizable.

This example may be found in \cite{KLO_ijrnc20} in the context of control affine systems with $n$ states and  $n-1$ inputs with the property that there exists $i=1,2,3$ such that 
\begin{equation}\label{3inprop:eq}
\dim\{ f_{1}, f_{2}, f_{3}, \ad{g} f_{i}\}=4,
\end{equation}
with 
$$f_{1} = \ddxx{1} + x_{3}\ddxx{2} +  x_{4} \ddxx{3}, \quad f_{2} = \ddxx{4}, \quad f_{3} = x_{1}\ddxx{3}, \quad g = u_{1}^{(0)} f_{1} + u_{2}^{(0)} f_{2} + u_{3}^{(0)} f_{3},$$
a property satisfied for $i=3$ since  
$$\ad{g} f_{3} = u_{1}^{(0)}  [f_{1},f_{3}] = u_{1}^{(0)} \left( \ddxx{3} - x_{1}\ddxx{2}\right) \not\in \{ f_{1}, f_{2}, f_{3}\}.$$ 
Therefore, in \cite[Theorem 3]{KLO_ijrnc20}, flat outputs are exhibited as 3 independent first integrals of $f_{3}$, yielding the prolongation $\dot{u}_{1}^{(1)}= u_{1}^{(2)}$, $\dot{u}_{2}^{(1)}= u_{2}^{(2)}$, with $\mbf{j}= (1, 1, 0)$ and thus $\mid \mbf{j} \mid = 2$. 

We show here that this prolongation is not minimal and we compute the minimal one. 

We initialize the algorithm by remarking that $\{ f_{2}, f_{3}\}$ is the unique maximal involutive subdistribution of $\{ f_{1}, f_{2}, f_{3}\}$, thus indicating that $u_{2}^{(1)}$ and $u_{3}^{(1)}$ are suitable candidates of non prolonged inputs.
We thus, consider the following vector fields, corresponding to prolongations of $u_{1}^{(1)}$, for all $l\geq 1$:
$$g_{0}^{(l,0,0)}\triangleq   u_{1}^{(0)} \left(\ddxx{1} + x_{3}\ddxx{2} + x_{4} \ddxx{3}\right) + u_{2}^{(0)} \ddxx{4} + u_{3}^{(0)} x_{1}\ddxx{3} + \sum_{p=0}^{l-1} u_{1}^{(p+1)}\ddu{1}{p},$$
$$g_{1}^{(l)}=\ddu{1}{l}, \qquad g_{2}^{(0)}=\ddu{2}{0}, \qquad g_{3}^{(0)}=\ddu{3}{0}.$$
The reader may easily check that $\Delta_{k}^{(l,0,0)} = \ol{\Delta_{k}^{(l,0,0)}}$ and $[\Gamma_{k}^{(l,0,0)}, \Delta_{k}^{(l,0,0)}] \subset \Delta_{k}^{(l,0,0)}$ for all $l\geq 1$ and $k=0,1,2$, and that $\dim \Delta_{2}^{(1,0,0)} = n+m = 7$. Therefore, the prolongation $\mbf{j}= (1,0,0)$, with $\mid \mbf{j} \mid =1$,  is minimal and the minimally purely prolonged system is equivalent by diffeomorphism and feedback to
$$ y_{1}^{(3)} = u_{1}^{(2)} , \qquad \ddot{y}_{2} = u_{2}^{(1)} \qquad 
y_{3}^{(3)}  = W(y_{1}, \ldots, y_{1}^{(3)}, y_{2}, \ldots, \ddot{y}_{2}, y_{3}, \ldots, \ddot{y}_{3} , u_{3}^{(1)}) $$
with flat outputs
$$y_{1} = x_{1},\quad  y_{2} = x_{4}, \quad  y_{3} = x_{2}$$  
and with
$$
\begin{array}{l}
\ds W(y_{1}, \ldots, y_{1}^{(3)}, y_{2}, \ldots, \ddot{y}_{2}, y_{3}, \ldots, \ddot{y}_{3} , u_{3}^{(1)}) 
\triangleq  y_{2}\dot{y}_{1}\left( \ddot{y}_{1} - \frac{\left(\dot{y}_{1}\right)^2}{y_{1}}\right) + 
\dot{y}_{2}\left(\dot{y}_{1}\right)^2 \\
\ds \hspace{2cm} 
+ y_{3}\frac{y_{1}^{(3)}}{\dot{y}_{1}} 
- \dot{y}_{3} \ddot{y}_{1}\left( \frac{1}{y_{1}} + 2\frac{ \ddot{y}_{1}}{( \dot{y}_{1})^2} \right)
 +  \ddot{y}_{3} \left( \frac{\dot{y}_{1}}{y_{1}} + 2 \frac{\ddot{y}_{1}}{\dot{y}_{1}}  \right) + y_{1}\dot{y}_{1}u_{3}^{(1)}.
\end{array}
$$


\section{Concluding Remarks}
We have established necessary and sufficient conditions for a system to be flat by pure prolongation, \ie belonging to the equivalence class of $0$ with respect to the \emph{equivalence by pure prolongation relation}. These conditions extend preliminary results of \cite{CLMscl,CLM,Sluis-scl93,SlT-scl96,BC-scl2004,FrFo-ejc2005} thanks to a thorough study of purely prolonged vector fields. 
We then deduce a computationally tractable algorithm giving the minimal prolongation in a finite number of steps using only Lie brackets and linear algebra. 

Possible extensions of this work towards general flatness necessary and sufficient conditions are under study.

\acknowledgement{The author wishes to express his warm thanks to Ph. Martin and Y. Kaminski for many fruitful discussions.}

 \bibliographystyle{amsplain}
\bibliography{Flatnew1}

\newpage
\begin{appendices}

\section{A Comparison Formula}\label{annex:compar:sec}
We establish a comparison formula expressing the iterated Lie brackets $\ad^{k}_{g_{0}^{(\mbf{j})}} g_{i}^{(j_{i})}$ for all $k\geq 0$, in terms of a combination of Lie brackets of the vector fields $g_{0}^{(\mbf{0})}= f\ddx$, $g_{1}^{(0)} = \ddu{1}{0}$, \ldots, $g_{m}^{(0)}=\ddu{m}{0}$ of the \emph{original system~\eqref{control-aff-sys:eq}}.

For simplicity's sake, we introduce a fictitious input $u_{0} \triangleq u_{0}^{(0)} \triangleq t$ so that $u_{0}^{(1)} =1$, $j_{0} =1$, and $u_{0}^{(k)}=0$ for all $k>1$. Thus, $\sum_{l=0}^{j_{0}-1} u_{0}^{(l+1)}g_{0}^{(\mbf{l})} = u_{0}^{(1)}g_{0}^{(\mbf{0})} = g_{0}^{(\mbf{0})}$, and 
$$g_{0}^{(\mbf{j})} = g_{0}^{(\mbf{0})} + \sum_{k=1}^{m}\sum_{l=0}^{j_{k}-1} u_{k}^{(l+1)} g_{k}^{(l)}
=  \sum_{k=0}^{m}\sum_{l=0}^{j_{k}-1} u_{k}^{(l+1)} g_{k}^{(l)}.$$ 


\begin{lem}\label{prepa1:lem}
For every $i=1,\ldots,m$, we have:
\begin{equation}\label{adkg0j-gij-:eq:appndx}
 \ad_{g_{0}^{(\mbf{j})}}^{k} g_{i}^{(j_{i})} =
(-1)^k  \frac{\partial}{\partial u_{i}^{(j_{i}-k)}} = (-1)^k g_{i}^{(j_{i}-k)}
\quad \mathrm{if~} k \leq  j_{i}, 
\end{equation}
and, for $k=j_{i}+\nu$, for all $\nu \geq 1$,

\begin{align}
&\ad_{g_{0}^{(\mbf{j})}}^{j_{i}+\nu} g_{i}^{(j_{i})} = (-1)^{j_{i}} \sum_{(r,\mbf{l},\mbf{p})\in \J^{(\mbf{j})}_{i,\nu}}
c_{\nu,\mbf{l}}^{\mbf{p}} u_{p_{r}}^{(l_{r})}  \cdots u_{p_{1}}^{(l_{1})} [g_{p_{r}}^{(0)} ,   \ldots, [g_{p_{1}}^{(0)}, g_{i}^{(0)} ] \ldots ] \label{adkg0j-gij+:eq:appndx}
\\
&= (-1)^{j_{i}} \ad_{g_{0}^{(\0)}}^{\nu} g_{i}^{(0)} 
+ \sum_{(q, r,\mbf{l},\mbf{p})\in \I^{(\mbf{j})}_{i,\nu}}
c_{\nu, q, \mbf{l}}^{\mbf{p}}
u_{p_{r}}^{(l_{r})} \cdots u_{p_{1}}^{(l_{1})} 
[g_{p_{r}}^{(0)} , \ldots,
[g_{p_{1}}^{(0)}, \ad_{g_{0}^{(\0)}}^{q} g_{i}^{(0)} ] \ldots ] \label{adkg0j-adq0gij+:eq:appndx}
\end{align}
where $\J^{(\mbf{j})}_{i,\nu}$ (resp. $\I_{i,\nu}$) is the set of multi-integers $(r,\mbf{l},\mbf{p}) \triangleq (r , l_{1}, \ldots, l_{r}, p_{1}, \ldots, p_{r})$ (resp. $(q, r , \mbf{l},\mbf{p}) \triangleq (q, r , l_{1}, \ldots, l_{r}, p_{1}, \ldots, p_{r})$) defined by 
\begin{align}
&\J^{(\mbf{j})}_{i,\nu} \triangleq \{(r , \mbf{l},\mbf{p}) \mid 1\leq r \leq \nu, \quad 0 \leq p_{\alpha} \leq m, ~1\leq l_{\alpha}\leq j_{\alpha},~ \alpha=1,\ldots, r,  ~ \;\mid \mbf{l} \mid = \nu \}, \label{ineq-indJ-adg:eq:appndx} 
\\
&(resp.\nonumber\\
&\I^{(\mbf{j})}_{i,\nu} \triangleq \left\{
\ds q,r \geq 1,~ 0\leq p_{\alpha} \leq m,~ 1\leq l_{\alpha}\leq j_{\alpha},~\alpha =1,\ldots, r, ~p_{r} \neq 0,~ r\leq~ \mid \mbf{l} \mid ~ = \nu-q \right\}) \label{ineq-indI-adg:eq:appndx}
\end{align}
with $\mid \mbf{l} \mid \triangleq \sum_{\alpha=1}^{r} l_{\alpha}$, $c_{\nu,\mbf{l}}^{\mbf{p}}\triangleq c_{\nu, l_{1}, \ldots, l_{r}}^{p_{1}, \ldots, p_{r}}\in \Z$ (resp. $c_{\nu, q, \mbf{l}}^{\mbf{p}} \triangleq c_{\nu, q, l_{1}, \ldots, l_{r}}^{p_{1}, \ldots, p_{r}}\in \Z$) and with the convention that the summation of the right-hand side of \eqref{adkg0j-gij+:eq:appndx} (resp. \eqref{adkg0j-adq0gij+:eq:appndx}) vanishes if $\J^{(\mbf{j})}_{i,\nu}= \emptyset$ (resp.  $\I^{(\mbf{j})}_{i,\nu}= \emptyset$), 
each bracket $[g_{p_{r}}^{(0)} ,   \ldots, [g_{p_{1}}^{(0)}, g_{i}^{(0)} ] \ldots ]$ in \eqref{adkg0j-gij+:eq:appndx}(resp. $[g_{p_{r}}^{(0)} , \ldots,
[g_{p_{1}}^{(0)}, \ad_{g_{0}^{(\0)}}^{q} g_{i}^{(0)} ] \ldots ]$ in \eqref{adkg0j-adq0gij+:eq:appndx}) depending at most on $\ol{x}^{(\mbf{0})}$.
%
%
\end{lem}
\begin{proof} By induction. 
It is immediately seen that
$$
 \ad_{g_{0}^{(\mbf{j})}} g_{i}^{(j_{i})} = [ f \frac{\partial}{\partial x} + \sum_{k=1}^{m}\sum_{l = 0}^{j_{k}-1} u_{k}^{(l+1)} \ddu{k}{l} , \ddu{i}{j_{i}}] 
 = - \ddu{i}{j_{i-1}} = - g_{i}^{(j_{i}-1)}
$$ 
Iterating this computation up to $k= j_{i}$ yields \eqref{adkg0j-gij-:eq:appndx}. In particular:
$$\ad^{j_{i}}_{g_{0}^{(\mbf{j})}} g_{i}^{(j_{i})} = (-1)^{j_{i}} \ddu{i}{0} = (-1)^{j_{i}} g_{i}^{(0)}.$$
Then, for $k=j_{i}+1$, using the fact that 
$[\ddu{k}{l} , \ddu{i}{0}] =0$ 
for all $i$, $k$ and $l\geq 0$, we have:
\begin{equation}
\label{ad-j+1g0-gj:eq:appndx}
\begin{aligned}
\ad_{g_{0}^{(\mbf{j})}}^{j_{i}+1} g_{i}^{(j_{i})} &= [g_{0}^{(\mbf{j})}, \ad_{g_{0}^{(\mbf{j})}}^{j_{i}} g_{i}^{(j_{i})}] = (-1)^{j_{i}} [g_{0}^{(\mbf{j})}, g_{i}^{(\0)}]\\
&= (-1)^{j_{i}} [ f \frac{\partial}{\partial x} 
+ \sum_{k=1}^{m}\sum_{l=0}^{j_{k}-1} u_{k}^{(l+1)} \ddu{k}{l} , \ddu{i}{0}] \\
& = (-1)^{j_{i}} [ f \frac{\partial}{\partial x} , \ddu{i}{0}] = (-1)^{j_{i}} \ad_{g_{0}^{(\mbf{0})}} g_{i}^{(0)}
\end{aligned}
\end{equation}
which proves that \eqref{adkg0j-gij+:eq:appndx} and  \eqref{adkg0j-adq0gij+:eq:appndx} hold at the order $k=j_{i}+1$, \ie $\nu = 1$, the summation of the right-hand side of \eqref{adkg0j-adq0gij+:eq:appndx} being equal to 0 since $\I^{(\mbf{j})}_{i,1} = \emptyset$.
Furthermore, a direct calculation shows that
\begin{equation}\label{A:gam-kij:eq}
(-1)^{j_{i}} \ad_{g_{0}^{(\mbf{0})}} g_{i}^{(0)} =
(-1)^{(j_{i}+1)} \frac{\partial f}{\partial u_{i}^{(0)}}(\ol{x}^{(\0)})\frac{\partial}{\partial x},
\end{equation} 
which proves that $\ad_{g_{0}^{(\mbf{j})}}^{j_{i}+1} g_{i}^{(j_{i})}\in \V X^{(\0)}$, where $\V X^{(\0)}$ is the vertical bundle of $X^{(\0)}$, 
\ie the set of vector fields that are linear combinations of $\frac{\partial}{\partial x_{1}}, \ldots, \frac{\partial}{\partial x_{n}}$, and whose coefficients are smooth functions that depend at most on  $\ol{x}^{(\mbf{0})}$ (but not of $u^{(l)}$ for all $l\geq 1$). It results that $[\frac{\partial}{\partial u_{p}^{(l)}} , \ad_{g_{0}^{(\mbf{0})}} g_{i}^{(0)}] = [g_{p}^{(l)} , \ad_{g_{0}^{(\mbf{0})}} g_{i}^{(0)}] =0$ for all $l\geq 1$, all $p=1,\ldots,m$, and all $i=1, \ldots,m$.

Assume now that \eqref{adkg0j-gij+:eq:appndx} and  \eqref{adkg0j-adq0gij+:eq:appndx} hold up to $k=j_{i}+\nu$, with $\nu>1$, 
and that all the brackets $[g_{p_{r}}^{(0)} ,   \ldots, [g_{p_{1}}^{(0)}, g_{i}^{(0)} ] \ldots ]$ and $[g_{p_{r}}^{(0)}, [ \ldots, [g_{p_{1}}^{(0)}, \ad_{g_{0}^{(\0)}}^{q} g_{i}^{(0)} ]\ldots ]]$ depend at most on   $\ol{x}^{(\mbf{0})}$. We immediately deduce that
$$[ \frac{\partial}{\partial u_{p_{r+1}}^{(l_{r+1})}} ,[g_{p_{r}}^{(0)}, [ \ldots 
[g_{p_{1}}^{(0)}, \ad_{g_{0}^{(\0)}}^{q} g_{i}^{(0)} ]\ldots ]]] =  [g_{p_{r+1}}^{(l_{r+1})} ,[g_{p_{1}}^{(0)}, [ \ldots 
[g_{p_{r}}^{(0)}, \ad_{g_{0}^{(\0)}}^{q} g_{i}^{(0)} ]\ldots ]]] = 0$$
for all $l_{r+1} \geq 1$, $p_{r+1} \in \{1,\ldots,m\}$,  $q \in \{0, \ldots,\nu - 1\}$,  $r\geq 1$, $p_{1}, \ldots, p_{r} \in \{0,\ldots, m\}$, and $i \in \{1, \ldots,m\}$. Thus:

$$
\begin{aligned}
& \ad_{g_{0}^{(\mbf{j})}}^{j_{i}+\nu+1} g_{i}^{(j_{i})} = [g_{0}^{(\mbf{j})},  \ad_{g_{0}^{(\mbf{j})}}^{j_{i}+\nu} g_{i}^{(j_{i})}]\\
&= (-1)^{j_{i}} \sum_{(r,\mbf{l},\mbf{p})\in \J_{i, \nu}^{(\mbf{j})}} [g_{0}^{(\mbf{j})}, c_{\nu, \mbf{l}}^{\mbf{p}}
u_{p_{r}}^{(l_{r})}  \cdots u_{p_{1}}^{(l_{1})}
[g_{p_{r}}^{(0)} ,   \ldots, [g_{p_{1}}^{(0)}, g_{i}^{(0)} ] \ldots ]]\\
&= (-1)^{j_{i}} \sum_{(r,\mbf{l},\mbf{p})\in \J_{i, \nu}^{(\mbf{j})}}  \sum_{p_{r+1}=0}^{m} 
c_{\nu, \mbf{l}}^{\mbf{p}}  u_{p_{r+1}}^{(1)}u_{p_{r}}^{(l_{r})}  \cdots u_{p_{1}}^{(l_{1})}
[g_{p_{r+1}}^{(0)},[g_{p_{r}}^{(0)} ,   \ldots, [g_{p_{1}}^{(0)}, g_{i}^{(0)} ] \ldots ]]\\
& \qquad +  (-1)^{j_{i}} \sum_{(r,\mbf{l},\mbf{p})\in \J_{i, \nu}^{(\mbf{j})}} \sum_{p_{r+1}=1}^{m} \sum_{l_{r+1}= 0}^{j_{p_{r+1}}} c_{\nu, \mbf{l}}^{\mbf{p}}
u_{p_{r+1}}^{(l_{r+1}+1)}\frac{\partial}{\partial u_{p_{r+1}}^{(l_{r+1})}} 
\left( u_{p_{r}}^{(l_{r})}  \cdots u_{p_{1}}^{(l_{1})}\right) 
[g_{p_{r}}^{(0)} ,   \ldots, [g_{p_{1}}^{(0)}, g_{i}^{(0)} ] \ldots ]\\
&\triangleq \sum_{(r,\mbf{\lambda},\mbf{\mu})\in \K_{i, \nu+1}^{(\mbf{j)}}} c_{\nu+1,\mbf{\lambda}}^{\mbf{\mu}} u_{\mu_{r+1}}^{(\lambda_{r+1})}\cdots  u_{\mu_{1}}^{(\lambda_{1})}[g_{\mu_{r+1}}^{(0)} ,   \ldots, [g_{\mu_{1}}^{(0)}, g_{i}^{(0)} ] \ldots ].
\end{aligned}
$$
Applying the Leibnitz rule to the penultimate line, using the fact that $\mid \mbf{l} \mid = \sum_{\alpha=1}^{r}l_{\alpha}= \nu$ by the induction assumption, we get  $1 + \mid \mbf{l} \mid = \sum_{\alpha=1}^{r+1} \lambda_{\alpha} = \nu+1$, with $\lambda_{\alpha} \geq 1,\; \alpha=1,\ldots, r+1$, and $ 0 \leq \mu_{\alpha} \leq m$, which proves that the latter formula holds for all $(r, \mbf{\lambda},\mbf{\mu})\in \K_{i,\nu+1}^{(\mbf{j})} \subset \J_{i,\nu+1}^{(\mbf{j})}$. We leave as an exercise to the reader the proof of the converse inclusion  $\J_{i,\nu+1}^{(\mbf{j})} \subset \K_{i,\nu+1}^{(\mbf{j})}$.  Therefore \eqref{adkg0j-gij+:eq:appndx}  is valid for all $\nu$. The proof of \eqref{adkg0j-adq0gij+:eq:appndx} follows exactly the same lines. The lemma is proven.
\end{proof}

\end{appendices}

%
%

\end{document}